\documentclass[11pt]{article}
\usepackage[latin1]{inputenc}
\usepackage[english]{babel}

\usepackage{amssymb}
\usepackage{amsmath,amsthm}%usepackage para numera\c{c}\~{a}o das f\'{o}rmulas
\usepackage[normalem]{ulem}

\usepackage{graphicx,color}

\numberwithin{equation}{section}    %numera\c{c}\~{a}o das f\'{o}rmulas

\newtheorem{thm}{Theorem}[section]

\newtheorem{cor}{Corollary}[section]
\newtheorem{lem}{Lemma}[section]
\newtheorem{rmk}{Remark}[section]
\newtheorem{dfn}{Definition}[section]
\newtheorem{exmp}{Example}[section]

\usepackage{amsmath,amssymb,amsfonts}

\newcommand{\R}{\mathbb{R}}

\newcommand{\dps}{\displaystyle}

\let \al=\alpha
\let \be=\beta
\let \var=\varphi
\let \vare=\varepsilon

\let \de=\delta
\let \th=\theta
\let \la=\lambda

\let \ga=\gamma

\let \q=\quad
\let \qq=\qquad
\let \med=\medskip

\let \dps =\displaystyle

\begin{document}

\vspace{0.5cm}
\begin{center}
  {\bf {\Large Attractivity of Saturated Equilibria for Lotka-Volterra Systems with Infinite Delays and Feedback Controls }}
\end{center}

\

 \centerline{\scshape Yoshiaki Muroya\footnote{Professor Yoshiaki Muroya passed away in October 2015, while the research for this paper was being conducted.
The second author wishes to dedicate this paper to his memory.} 
}
\medskip
{\footnotesize
 \centerline{Department of Mathematics, Waseda University}
 \centerline{ 3-4-1 Ohkubo, Shinjuku-ku, Tokyo 169-8555, Japan}
%\centerline{ymuroya@waseda.jp}}

\

\centerline{\scshape Teresa Faria\footnote{Corresponding author.
E-mail:~teresa.faria@fc.ul.pt.}}
%To the memory of Professor Yoshiaki Muroya, who passed way in October 2015, while we were finishing this %manuscript.}}
\medskip
{\footnotesize
 \centerline{Departamento de Matem\'atica and CMAF-CIO,}
   \centerline{ Faculdade de Ci\^encias, Universidade de Lisboa}
   \centerline{Campo Grande, 1749-016 Lisboa, Portugal}
%\centerline{teresa.faria@fc.ul.pt}}

\

\vskip0.3in

%{\it \centerline{.}}

\bigskip

%The abstract of your paper
\begin{abstract} In this paper, we apply a Lyapunov functional approach to Lotka-Volterra systems with infinite delays and feedback controls and establish that the feedback controls  have no influence on the attractivity properties of a saturated equilibrium. This improves previous results by the authors and others, where, while feedback controls were used mostly to change the position of a unique saturated equilibrium, additional conditions involving the controls had to be assumed in order to preserve its global attractivity. The  situation of partial extinction is further analysed,  for which the original system is reduced to a lower dimensional one  which maintains  its   global dynamics features. 
\end{abstract}

\vskip 2mm



\section{Introduction}
\setcounter{equation}{0}
Recently, there has been a significant number of publications on the study of Lotka-Volterra population models with delays and feedback controls, see  \cite{Chen:2007,  Gopalsamy:2003,HTG, Li:2013,Muroya:2014, Muroya:2015, Nie:2009,Shi:2012,SY:16,WTJ:08,Zhang:2008} and references cited therein. In particular, 
%Fan and Wang \cite{Fan:2010} investigated the stability  for the one species case by monotone techniques;
 Li {\it et al.} \cite{Li:2013} established  results on the influence of the feedback controls on extinction and global attractivity of a two-species autonomous competitive Lotka-Volterra system with infinite delays by applying Lyapunov functional techniques. Motivated by this work and others \cite{Faria:2010,Muroya:2003}, Faria and Muroya \cite{Faria:2015} considered the following multiple species Lotka-Voltera models with infinite delays and feedback controls:
\begin{equation}\label{eq}
\left\{
\begin{array}{ll}
\displaystyle{x_i^{\prime}(t)=x_i(t)\biggl(b_i-\mu_i x_i(t)-\sum_{j=1}^n a_{ij} \int_0^{\infty}K_{ij}(s)x_j(t-s)\, ds} \\
\ \ \ \ \ \ \ \ \ \ \displaystyle{-c_i \int_0^{\infty}G_i(s)u_i(t-s)\, ds\biggr)}, \\
\displaystyle{u_i^{\prime}(t)=-e_i u_i(t)+d_i x_i(t), \quad i=1,\ldots,n}, \\
\end{array}
\right.
\end{equation}
where $ \mu_i,  c_i , d_i, e_i$ are positive constants, $b_i,a_{ij}\in\R$, and
the kernels $K_{ij}, G_i: [0,\infty) \to [0,\infty)$ are $L^1$ functions, normalized so that 
\begin{equation}\label{1.2}
\displaystyle{\int_0^{\infty} K_{ij}(s)\, ds=1, \ \int_0^\infty G_i(s)\, ds=1, \ \mbox{for} \ i,j=1,\ldots,n}.
\end{equation} 
In \eqref{eq}, $x_i(t)$ denotes the density of  an $i$th-species or class population, $u_i(t)$ is a feedback control variable, 
$b_i$ is the intrinsic  growth rate, $\mu_i$  is a self-limitation coefficient
  in the instantaneous negative feedback term,
$a_{ij}$ are the intra- (if $i=j$) and inter-specific (if $i\ne j$) cooperative/competition coefficients,
   for $i,j=1,2, \ldots,n$. 
Actually, the particular case of Lotka-Volterra systems  \eqref{eq} without controls, or with controls only for some of the variables -- i.e., the situation with  $c_i\ge 0$ for $i\in\{1,\dots,n\}$ -- can be included in most  of the  results presented here. For simplicity, we also assume the following technical condition:
 for any $i,j \in \{1,\ldots,n\}$, 
\begin{equation}\label{1.3}
\int_0^{\infty} s K_{ij}(s)\, ds < \infty\quad \mbox{and} \quad \int_0^{\infty} s G_{i}(s)\, ds < \infty.
\end{equation}
As will be shown in the proofs, 
 when using Lyapunov functional techniques, condition  \eqref{1.3} allows
the treatment of the infinite delay terms  in a way  analogous to the one usually employed to deal with the  finite delay case. 
For a more general setting which does not require  \eqref{1.3}, see Faria and Muroya \cite{Faria:2015}.

   Since some previous works by Kuang and Smith \cite{Kuang:95,KuangS:93}, the literature on  Lotka-Volterra with infinite delays has been quite vast, and it is impossible to mention here all the significant works. For some results on stability, extinction and permanence for autonomous and non-autonomous   Lotka-Volterra models with infinite delay (and no controls), see e.g. \cite{BG:97,Chen:2015,Faria:2010,Faria:2014, Faria:2016,Hou,MOca:12, MOca:15, MOca:06,TangZou:2003, TC:2001, TR:2006}. To treat the permanence of non-autonomous Lotka-Volterra systems, methods inspired in the setting proposed by Ahmad and Lazer (see e.g.~\cite{AL98, AL05}) and initiated in the classical work of Vance of Coddington \cite{VC:89}, have proven to be very fruitful. Here, non-autonomous Lotka-Volterra systems are not treated {\it per se}, but  as a secondary outcome  of the method developed, see Remark \ref{rmk3.2}.

An admissible  initial condition for  \eqref{eq} takes the form
\begin{equation}\label{initial}
\begin{split}
&x_i(\theta)=\varphi_i(\theta), u_i(\th)=\psi_i(\th), \quad \theta \in (-\infty,0], \\
& \varphi_i(0)>0,\psi_i(0)>0 \quad i=1,\ldots,n,
\end{split}
\end{equation}
where $\varphi_i,\psi_i, \ i=1,\ldots,n$, are non-negative, bounded and continuous functions on $(-\infty,0]$. 
For the initial value problems \eqref{eq}-\eqref{initial}, from the general theory for delay differential equations (DDEs) it follows that solutions are positive and defined for all $t>0$. Without loss of generality, although not relevant here, we may suppose that, for all $i$, the linear operators defined by 
$L_{ii}^K(\var):=\int_0^{\infty} K_{ii}(s)\var (-s)\, ds$ and $L_{i}^G(\var):=\int_0^{\infty} G_{i}(s)\var (-s)\, ds,$
 for $\var:(-\infty,0]\to\R$ continuous and bounded, are non-atomic at zero, i.e., $K_{ii}(0)=K_{ii}(0^+)$ and $G_{i}(0)=G_{i}(0^+)$ \cite{HaleLunel}.

   For system \eqref{eq}, we define the matrices
\begin{equation}\label{1.5}
\begin{split}
M_0&=[\delta_{ij} \mu_i+a_{ij}]_{n \times n}, \\
M&=[\delta_{ij} \la_i+a_{ij}]_{n \times n}, \quad{\rm where}\q  \lambda_i=\mu_i+(c_i d_i)/e_i, \quad i=1,\ldots,n, 
\end{split}
\end{equation}
which are called the  {\bf  community matrix} and the {\bf controlled community matrix}, respectively. Here and  throughout the paper, the standard notation $\delta_{ij}=1$ if $i=j$ and $\delta_{ij}=0$ if $i \neq j$ is used.

%For \eqref{eq}, we define the $n\times n$ {\it community matrices}, without and with controls, given by
%$$M_0=[\delta_{ij} \mu_i+a_{ij}]$$
%and
%$$M=[\delta_{ij} (\mu_i+\frac{c_i e_i}{d_i})+a_{ij}],$$
%respectively, where $\delta_{ij}=1$, if $i=j$ and $\delta_{ij}=0$, if $i \neq j$. 

% The present research was strongly motivated by \cite{Faria:2015,Li:2013}.
 In this paper, the authors pursue their work in
 \cite{Faria:2015}, where, among other results,   general sufficient conditions for the existence and global attractivity of a saturated equilibrium for model \eqref{eq} were established.  The definition of a  {\it saturated equilibrium} can be found in \cite{Hofbauer:1988} and will be recalled in Section 2. Of course, if a saturated equilibrium $E^*$ of \eqref{eq} with exactly $p$ positive components  $x_i^*>0$, corresponding to the $i$th-populations $x_i(t)$,  is a global attractor of all positive solutions, this means that  those $p$ species $x_i(t)$ stabilize with time at a constant value $x_i^*$, whereas the other $n-p$ populations are driven to extinction.

   While  the literature usually   treats the case of competitive systems only -- which amounts to having $a_{ij}\ge 0$ for $i\ne j$ in \eqref{eq} --, here, as  in  \cite{Faria:2015},  we shall not impose any  restrictions on the signs of the intra- and inter-specific  coefficients  $a_{ij}$, nor on the Malthusian rates  $b_i$.    In \cite{Faria:2015}, we assumed a form of diagonal dominance of the instantaneous negative intra-specific terms over the infinite delay effect, in both the population variables and controls.
  To be more precise, the main result  in  \cite{Faria:2015} states that if the $n \times n$ matrix $\hat{M}:=[\delta_{ij} \mu_i-|a_{ij}|-\delta_{ij}\frac{c_i e_i}{d_i}]$ is a nonsingular $M$-matrix (see \cite{Berman:1979}  and Section 2 for a definition), then there exists a  saturated equilibrium of \eqref{eq} which  is globally attractive. When the saturated   equilibrium is positive, in fading memory spaces this condition also implies its  asymptotic stability,  see  \cite[Theorem 3.8]{Faria:2015}. When the saturated equilibrium is on the boundary of $\R^{2n}_+$, sharper criteria for the extinction of part of the populations were also given in \cite{Faria:2015}.%, which turn out to be particularly useful for predator-prey models. 

With the present research,  better criteria for the global attractivity of a saturated equilibrium $E^*$ than the ones in   \cite{Faria:2015} are achieved. The present techniques  are quite different from the ones in \cite{Faria:2015}, where the main results were obtained by a monotone flow approach and depended on the  feedback controls. In contrast, here we construct original Lyapunov functionals to deduce our main criteria of global attractivity. Namely, we derive that the saturated equilibrium $E^*$ of \eqref{eq} is globally attractive if the matrix $\hat{M}_0:=[\delta_{ij} \mu_i-|a_{ij}|]_{n \times n}$ is a nonsingular $M$-matrix (although a better result is obtained if some components of $E^*$ are zero).
%for global attractivity, which states that the saturated equilibrium of \eqref{eq} is globally attractive if the matrix $\hat{M}_0:=[\delta_{ij} \mu_i-|a_{ij}|]_{n \times n}$ is a nonsingular $M$-matrix.
This imposition  does not depend on the control coefficients and is much less restrictive than  having $\hat{M}$ a non-singular M-matrix.
Therefore, we conclude that, whereas feedback controls can effectively  be useful to change the position of a unique saturated equilibrium $E^*$ of \eqref{eq}  -- as it was well-illustrated in \cite{Faria:2015} --, they have no influence on the attractivity properties of  $E^*$. Moreover, weaker
 sufficient conditions for partial extinction will be given: we shall  show that the global attractivity can be reduced to  the one for a reduced system, i.e., a system of smaller dimension, in the spirit in Muroya \cite[Theorem 1.2]{Muroya:2003} and  Shi {\it et al.} \cite[Theorem 3.1]{Shi:2012}.
 
Integro-differential equations with infinite delay have been considered in population dynamics since the times of Volterra, in order to account for the entire past of the species.
Dealing with DDEs with infinite delays requires  a careful choice of an {\it admissible} (in the sense of Hale and Kato, see \cite{HaleKato,HMN}) Banach phase space. The additional property of being a  {\it fading memory space} is important, in order to recover some classical results, such as the  principle of linearized stability and precompactness of bounded periodic orbits. A rigorous  theoretical framework to deal with Lotka-Volterra systems with infinite delays was provided in e.g.  \cite{Faria:2010, Faria:2016,Faria:2015}. In order to avoid repetitions,  here we shall not present a suitable phase-space for \eqref{eq}, nor an abstract formulation of the initial value problem (IVP)  \eqref{eq}-\eqref{initial}: in summary,  we say that for the IVP \eqref{eq}-\eqref{initial}
 existence, uniqueness  and continuation of solution for $t\ge 0$ is well-established, and address  the reader to the literature.

The contents of this paper are organized as follows. Section 2 is a section of preliminaries, where we recall the definition of a saturated  equilibrium, give  basic conditions for such an equilibrium  to be positive or  on the boundary of  $\R_+^{2n}$, and summarize
 some results  from \cite{Faria:2015}. In Section 3, the main result on global attractivity of the  saturated  equilibrium of \eqref{eq} is proven. In Section 4, sharper sufficient conditions
 for partial extinction are further analysed. To illustrate the theoretical results, the paper finishes with some simple examples.

%%%%%%%%%%%%%%%%%%%%%%%%%%%%%% Section 2 %%%%%%%%%%%%%%%%%%%%%%%%%%%%
\section {Preliminaries and basic results on saturated equilibria}
\setcounter{equation}{0}

Throughout the paper,  for the controlled Lotka-Volterra system \eqref{eq} the following general hypothesis is assumed:

\begin{itemize}
\item[(H0)] $ \mu_i,  d_i, e_i$ are positive constants, $c_i\ge 0, b_i,a_{ij}\in\R$, 
the kernels $K_{ij}, G_i: [0,\infty) \to [0,\infty)$ are in $L^1[0,\infty)$ and satisfy \eqref{1.2} and \eqref{1.3}.
\end{itemize}

For most cases, we are interested in the situation with effective controls, i.e., with $c_i>0$ for all $i$, but the situation without part or all of the control variables $u_i(t)$ is allowed.
In the absence of controls, the  Lotka-Volterra system reads as
%
%\left\{
%\begin{array}{ll}
\begin{equation}\label{eqUn} x_i^{\prime}(t)=x_i(t)\biggl(b_i-\mu_ix_i(t)-\sum_{j=1}^n a_{ij} \int_0^{\infty}K_{ij}(s)x_j(t-s)\, ds\biggr),\q i=1,\dots,n.
\end{equation}

 Clearly, the introduction of controls might change  the dynamics of \eqref{eqUn}. In \cite{Faria:2015}, the  controls were mainly used to change the position of a globally attractive equilibrium, and further requirements on the controls were imposed in order to preserve its attractivity. Here, we shall show that in fact the controls do not have any effect on the attractivity of the saturated equilibrium, therefore additional restrictions are not needed.

%%%%%%%

A point $E^*=(x_1^*,u_1^*,\dots, x_n^*,u_n^*)\in \R^{2n}$ is an equilibrium of \eqref{eq} if and only if
$$x_i^*=0\ {\rm or}\ (M x^*)_i=b_i, \q {\rm and} \q u_i^*={{d_i}\over {e_i}}x_i^*,\q i=1,\dots,n.$$
In view of the biological interpretation of the model, only non-negative solutions are meaningful, thus only solutions with initial conditions \eqref{initial} will be considered. 
The following definition of a {\it saturated equilibrium} can be found in e.g. \cite{Hofbauer:1988,Faria:2015}.

\med

\begin{dfn}\label{dfn2.1}   Let $E^*=(x_1^*,u_1^*,\dots,x_n^*, u_n^*)$ be an equilibrium of \eqref{eq}. We say that $E^*$ is a {\it saturated equilibrium} if $E^*$ is non-negative and $x^*=(x_1^*,\dots,x_n^*)$ satisfies
\begin{equation}\label{2.1}
(Mx^*)_i\ge b_i\q {whenever}\q x_i^*=0,\q i=1,\dots,n.
\end{equation}
A saturated equilibrium $E^*\ge 0$ of \eqref{eq} is said to be {\it globally attractive} if  it attracts all solutions  of the problems \eqref{eq}-\eqref{initial}, i.e.,  $X(t)\to E^*$ as $t\to\infty$ for all positive solutions $X(t)=(x_1(t),u_1(t),\dots ,x_n(t),u_n(t))$ of \eqref{eq}.
\end{dfn}

For \eqref{eqUn} and \eqref{eq}, define the  {\it community matrix} $M_0$ and the {\it controlled community matrix} $M$, respectively, as in \eqref{1.5}. Consider also  the ${n \times n}$ matrices
\begin{equation}\label{2.3}
M_0^-=[\delta_{ij} \mu_i-a_{ij}^-],\q \hat M_0=\Big[\delta_{ij} \mu_i- |a_{ij}|\Big ], 
\end{equation}
where $a_{ij}^-=0$ if $a_{ij}\ge 0$, $a_{ij}^-=-a_{ij}$ if $a_{ij}< 0$.

As for ordinary differential equation (ODE) models, the algebraic properties of $M_0$ and $M$ determine many features of  the asymptotic behaviour of solutions to \eqref{eqUn} and \eqref{eq}.  (cf.~e.g.~\cite{Yuan:17,Faria:2014, Hofbauer:1988}). 
The existence of a saturated equilibrium depends on the properties of the controlled community matrix $M$. Further properties of some special matrices will be used for the analysis of  the attractive properties of equilibria.
The concepts of P-matrix and  M-matrix \cite{Berman:1979, Hofbauer:1988}, given below,  are crucial for results and arguments used in this paper.

\begin{dfn}\label{dfn2.2}  Let $B=[b_{ij}]$ be an ${n\times n}$ matrix. We say that $B$ is   a {\bf P-matrix} if all its principal minors are positive.
For $B$  with $b_{ij}\le 0$ for $i\ne j$, $B$ is said to be an {\bf  M-matrix} (respectively a {\bf non-singular M-matrix}) if   all its principal minors are non-negative (respectively positive).
 \end{dfn}
 
%The reader should also be aware that some authors often use the term {\it M-matrix},  to designate a matrix  defined here as a  {\it non-singular M-matrix}.   
For $B$ a square matrix with non-positive off-diagonal entries, it is well-known that $B$ is an M-matrix if and only if all its eigenvalues   have non-negative  real parts. The following properties of  non-singular M-matrices, as well as additional ones, can be found in   \cite{Berman:1979}.  A further property will be given in Section 3.

% and \cite{Fiedler:1986} (where these matrices are called {\it matrices of class}  $K$)    
\begin{lem}\label{lem2.1}  Let $B=[b_{ij}]$ be an ${n\times n}$ matrix with $b_{ij}\le 0$ for $i\ne j$.
%\vskip 1mm
%(a) The following assertions are equivalent:
%(i) $B$ is an M-matrix; (ii) all principal minors of $B$ are non-negative.\vskip 1mm
%(b) 
The following assertions are equivalent:  
(i) $B$ is a non-singular M-matrix; (ii) $B$ is an M-matrix and $\det B\ne 0$; (iii) all eigenvalues of $B$ have positive real parts; (iv)  there exists a positive vector $v$ such that $Bv>0$.%; (vi) $\det B\ne 0$ and $B^{-1}\ge 0$.
\end{lem}

We recall some results established previously by the authors in  \cite{Faria:2015}, where it was assumed that  $ \mu_i,  d_i, e_i$ are positive constants, $c_i\ge0, b_i,a_{ij}\in\R$, and
 $K_{ij}, G_i: [0,\infty) \to [0,\infty)$  satisfy \eqref{1.2}. One should once more emphasize  that, contrary to what is often assumed in the literature,  here  the coefficients $b_i,a_{ij}$ have no prescribed signs.

\begin{thm}\label{thm2.1}\cite{Faria:2015} Consider the controlled system \eqref{eq}, and define the matrices $M,M_0$ and $M_0^-$ as in \eqref{1.5} and  \eqref{2.3}.\vskip 0mm
(i) If $M_0^-$ is a non-singular M-matrix, then  all solutions of \eqref{eq} with initial conditions \eqref{initial} are defined on $[0,\infty)$ and \eqref{eq} is dissipative;  i.e., there exists a uniform upper bound $K>0$ such that all solutions of \eqref{eq}-\eqref{initial} satisfy $\limsup_{t\to \infty} x_i(t)\le K, \limsup_{t\to \infty} u_i(t)\le K$ for $1\le i\le n$. \vskip 0mm
(ii) If $M$ is a P-matrix,   there is a unique saturated equilibrium $E^*$ of \eqref{eq}.\end{thm}

Several interesting conclusions can be drawn from the above  theorem: for instance,   a  competitive system \eqref{eq}, i.e., when $a_{ij}\ge 0$ for all $i\ne j$, is dissipative if $\mu_i-a_{ii}^->0$ for all $i$. 

\begin{thm}\label{thm2.2}\cite{Faria:2015} (i) Suppose that $M$ is a P-matrix.
The trivial solution $(0,0,\ldots,0,0)$  is the saturated equilibrium of \eqref{eq} if and only if  
$b_i \leq 0, \, i=1,\dots,n.$
Furthermore, if the $n \times n$ matrix $M_0^-=[\delta_{ij} \mu_i-a_{ij}^-]_{n \times n}$ is an M-matrix, then the trivial solution of \eqref{eq} is globally attractive.  \vskip 0cm
(ii)  If $c_i>0$ for all $i$ and the $n \times n$ matrix 
\begin{equation}\label{Mhat}
\hat{M}=\Big[\delta_{ij} \mu_i-|a_{ij}|-\delta_{ij}\frac{c_i d_i}{e_i}\Big]
\end{equation} is an M-matrix, then there exists a saturated equilibrium $E^*$, which is  a global attractor of all  positive solutions of \eqref{eq}.  If in addition $E^*$ is positive and $\hat M$ is non-singular, then 
$E^*$ is globally asymptotically stable.
\end{thm}

The main goal of this paper is to improve the criterion for the global attractivity of the saturated equilibrium given  above in Theorem \ref{thm2.2}(ii).

We start by establishing more precise conditions on each  equilibrium to be saturated.
Consider the $n \times n$ matrix $A=[a_{ij}]$
 and denote
\begin{equation}\label{aii}
\hat{a}_{ii}=\lambda_i+a_{ii}, 
\end{equation}
where the coefficients $ \lambda_i=\mu_i+(c_i d_i)/e_i$ are as in \eqref{1.5}. For each $p=1,\dots,n$, set
\begin{equation}\label{R0}
R_0^p=
\left|
\begin{array}{cccc}
\hat{a}_{11}& a_{12}&\cdots&a_{1p} \\
a_{21}& \hat{a}_{22}&\cdots&a_{2p} \\
\vdots&\vdots&\ddots&\vdots \\
a_{p1}& a_{p2}&\cdots&\hat{a}_{pp} \\
\end{array}
\right|, 
\end{equation}
and, for $i=1,\ldots,p$, 
\begin{equation}\label{Ri}
R_i^p=\left|
\begin{array}{ccccccc}
\hat{a}_{11}&\cdots& a_{1,i-1}&b_1&a_{1,i+1}&\cdots&a_{1p} \\
\vdots&\ddots&\vdots&\vdots&\vdots&\ddots&\vdots \\
a_{i-1,1}&\cdots& \hat{a}_{i-1,i-1}&b_{i-1}&a_{i-1,i+1}&\cdots&a_{i-1,p} \\
a_{i,1}&\cdots&a_{i,i-1}&b_i&a_{i,i+1}&\cdots&a_{i,p} \\
a_{i+1,1}&\cdots& a_{i+1,i-1}&b_{i+1}&\hat{a}_{i+1,i+1}&\cdots&a_{i+1,p} \\
\vdots&\ddots&\vdots&\vdots&\vdots&\ddots&\vdots \\
a_{p1}&\cdots& a_{p,i-1}&b_p&a_{p,i+1}&\cdots&\hat{a}_{pp} \\
\end{array}
\right|.
\end{equation}

Observe that $R_0^p> 0,\, p=1,\dots,n,$ if $M$ is a P-matrix. If $1\le p<n$, for any fixed $q\in \{ p+1,\dots, n\}$, we also define 
\begin{equation}\label{tildeR-q-0}
R_q^{p+1,q}=\left|
\begin{array}{ccccc}
\hat{a}_{11}&a_{12}&\cdots&a_{1p}&b_1 \\
a_{21}& \hat{a}_{22}&\cdots& a_{2p} &b_2\\
\vdots&\vdots&\ddots&\vdots &\ddots\\
a_{p1}& a_{p2}&\cdots&\hat {a}_{pp}&b_p\\
a_{q1}& a_{q2}&\cdots&a_{qp}&b_q\\
\end{array}
\right| ,
\end{equation}
and remark that
\begin{equation}\label{tildeR-q}
R_q^{p+1,q}=b_q R_0^p -\sum_{j=1}^p a_{qj}R_j^p ,\q q=p+1,\dots, n. 
\end{equation}

First, we investigate the existence of a positive equilibrium of \eqref{eq}.  
A positive equilibrium $E^*=(x_1^*,u_1^*,x_2^*,u_2^*,\ldots,x_n^*,u_n^*)$ of system \eqref{eq} must satisfy
\begin{equation}
\displaystyle{b_i-\lambda_i x_i^*-\sum_{j=1}^n a_{ij} x_j^*=0, \quad -e_i u_i^*+d_i x_i^*=0, \quad i=1,\ldots,n}.
\end{equation}
By  Cramer's formulas,  $E^*$ is a positive equilibrium of system \eqref{eq} 
if and only if
$
 R_i^n>0$ for $i=0,1,\ldots,n,$ or   $R_i^n<0$ for $i=0,1,\ldots,n.$
In this case, 
\begin{equation}\label{Cramer}
x_i^*=\frac{R_i^n}{R_0^n}>0, \quad u_i^*=\frac{d_i}{e_i}x_i^*>0, \quad i=1,\ldots,n.
\end{equation}
If $M$ is a P-matrix, the saturated equilibrium of \eqref{eq} is positive if and only if $R_i^n>0$ for all $i$.

%one  obtains the following result.
% 
%%%%%%%%%%%%%%%%%%%%%%%%%%%%%%%%% Lemma 2.2 %%%%%%%%%%%%%%%%%%%%%%%%%%%%%%%%%%%%
%\begin{lem}\label{lem2.2}
%There exists a unique positive equilibrium $E^*=(x_1^*,u_1^*,x_2^*,u_2^*,\ldots,x_n^*,u_n^*)$ of system \eqref{eq} 
%if and only if
%\begin{equation}\label{asump-Ri+}
% R_i^n>0, \ i=0,1,\ldots,n, \quad \mbox{or} \quad  R_i^n<0, \ i=0,1,\ldots,n.
%\end{equation}
%In this case, 
%\begin{equation}\label{Cramer}
%x_i^*=\frac{R_i^n}{R_0^n}>0, \quad u_i^*=\frac{d_i}{e_i}x_i^*>0, \quad i=1,\ldots,n.
%\end{equation}
%If $M$ is a P-matrix, then the saturated equilibrium of \eqref{eq} is positive if and only if $R_i^n>0, \ i=0,1,\ldots,n$. 
%\end{lem}

%For this situation, in Theorem \ref{thm2.2} (\cite{Faria:2015}) tells us that the  positive equilibrium is globally asymptotically stable (GAS or GA?) if the matrix $\hat{M}$ in \eqref{Mhat}
% is a nonsingular $M$-matrix.

%----
%
%
% A priori, there are at most $2^n$  equilibria of \eqref{eq}, and we will give more precise conditions on the existence of each one, in order to analyse which  one is the saturated equilibrium. 
%
%
%
%-------

The case of a nontrivial saturated equilibrium on the boundary of the positive cone $R_+^n=[0,\infty)^n$ is  now analysed. Although there may be $2^n$ possible nonnegative equilibria of \eqref{eq}, for simplicity, we reorder the variables and 
 restrict our attention only to equilibria of the form
\begin{equation}\label{Ep}
E^{*,p}=(x_1^{*,p},u_1^{*,p},\ldots,x_p^{*,p},u_p^{*,p},0,0,\ldots,0,0),
\end{equation}
$\displaystyle{x_l^{*,p}>0}$, $ \dps {u_l^{*,p}=\frac{d_l}{e_l}x_l^{*,p}}$ for $l=1,\ldots,p, \ p\in \{0,1,\ldots,n\}$. The case $p=0$ (of the trivial equilibrium) has already been addressed in \cite{Faria:2015}, see Theorem \ref{thm2.2}(i). 

%As announced, here we shall give sharper conditions for the global attractivity of a  nontrivial saturated equilibrium, thus we now analyse the situations of $p>0$.

With this notation, if $p\in \{1,\ldots,n-1\}$ and $E^{*,p}$ is the saturated equilibrium of \eqref{eq}, together with the original system,  we shall also consider the following {\it reduced and rearranged} $p$-species Lotka-Volterra system with feedback controls and infinite delay:
\begin{equation}\label{reduced-eq}
\left\{
\begin{array}{ll}
\displaystyle{x_i^{\prime}(t)=x_i(t)\biggl(b_i-\mu_i x_i(t)-\sum_{j=1}^p a_{ij} \int_0^{\infty}K_{ij}(s)x_j(t-s)\, ds} \\
\ \ \ \ \ \ \ \ \ \ \displaystyle{-c_i \int_0^{\infty}G_i(s)u_i(t-s)\, ds\biggr)}, \\
\displaystyle{u_i^{\prime}(t)=-e_i u_i(t)+d_i x_i(t), \quad i=1,\ldots,p}. \\
\end{array}
\right.
\end{equation}
It is apparent that, for each $p \in \{1,\ldots,n-1\}$,  
%$E^{*,p}=(x_1^{*,p},u_1^{*,p},\ldots,x_n^{*,p},u_n^{*,p})$
$E^{*,p}=(x_1^{*,p},u_1^{*,p},\ldots,x_p^{*,p},u_p^{*,p},0,0,\ldots,0,0)$ with $x_i^{*,p}>0, 1\le i\le p,$ is an equilibrium of \eqref{eq} if and only if  $\tilde{E}^{*,p}:=(x_1^{*,p},u_1^{*,p},\ldots,x_p^{*,p},u_p^{*,p})$ is 
a positive equilibrium 
% $\tilde{E}^{*,p}:=(x_1^{*,p},u_1^{*,p},\ldots,x_{p}^{*,p},u_{p}^{*,p})$ 
of its reduced and rearranged system \eqref{reduced-eq}. 
%We refer to the positive equilibrium $\tilde{E}^{*,p}$ of \eqref{reduced-eq} as a ``{\it reduced equilibrium of $E^{*,p}$ of \eqref{eq} into \eqref{reduced-eq}}"; inversely, we refer to  a saturated equilibrium $E^{*,p}$ of \eqref{eq} as an "{\it extended equilibrium of the positive equilibrium $\tilde{E}^{*,p}$ of \eqref{reduced-eq} into \eqref{eq}}".
The relations between an equilibrium $E^{*,p}$ of \eqref{eq} and the positive equilibrium  $\tilde{E}^{*,p}$ of its reduced and rearranged system will be deeper exploited in Section 4.

A basic result on a saturated equilibrium is as follows.
\begin{lem}\label{lem2.3} 
Let  $M$ be a P-matrix and $p\in\{1,\dots,n-1\}$. Then, there exists a unique positive equilibrium  
$\tilde{E}^{*,p}:=(x_1^{*,p},u_1^{*,p},\ldots,x_{p}^{*,p},u_{p}^{*,p})$ of the reduced and rearranged system \eqref{reduced-eq} if and only if
 $R_i^p>0, \, i=0,1,\ldots,p.$
In this case, 
\begin{equation}\label{tildeCramer}
x_i^{*,p}=\frac{R_i^p}{R_0^p}>0, \quad u_i^{*,p}=\frac{d_i}{e_i}x_i^{*,p}>0, \quad i=1,2,\ldots,p.
\end{equation}
Moreover, $E^{*,p}=(x_1^{*,p},u_1^{*,p},\ldots,x_{p}^{*,p},u_{p}^{*,p},0,0,\dots,0,0)$   is the saturated equilibrium of \eqref{eq} if and only if
\begin{equation}\label{satu-cond-k}
\left\{
\begin{array}{ll}
R_i^p>0, \ & \mbox{for any} \ i=1,\ldots,p, \\
R_q^{p+1,q} \leq 0, \ & \mbox{for any} \ q =p+1,\dots, n. 
\end{array}
\right.
\end{equation}
In this case,
\begin{equation}\label{b-q}
b_q\le \sum_{j=1}^p a_{qj}x_j^{*,p},\q q =p+1,\dots, n.\end{equation}
%A similar result holds  if $R_0^p<0$.
%
%, then $E^{*,p}$ is a saturated equilibrium of \eqref{eq} if and only if
%\begin{equation}\label{satu-cond-k(-)}
%\left\{
%\begin{array}{ll}
%R_j^p<0, \ & \mbox{for any} \ j=1,\ldots p, \\
%R_{q}^{p+1,q} \geq 0, \ & \mbox{for any} \ q=p+1,\ldots,n. 
%\end{array}
%\right.
%\end{equation}
\end{lem}

\begin{proof} The first part is apparent. %Now, let ${R_0^p}>0$.
The first inequalities in \eqref{satu-cond-k} imply \eqref{tildeCramer}.
 By \eqref{tildeR-q} and the last inequalities in  \eqref{satu-cond-k}, it is clear that  \eqref{2.1} is satisfied.
 % The case $R_0^p<0$  is similar.  
 \end{proof}

\begin{exmp}\label{exmp2.1} {\rm
For \eqref{eq} with $n=2$, assume that $M$ is a P-matrix, that is,
\begin{equation}\label{M0-matrix-2}
\left\{
\begin{array}{ll}
\la_i+a_{ii}>0, \q i=1,2,  \\
\det M:=(\la_1+a_{11})(\la_2+a_{22})-a_{12} a_{21}>0.
\end{array}
\right.
\end{equation}
From Theorem \ref{thm2.2} and Lemma \ref{lem2.3}, we have: \\
(i) $E^{*,0}=(0,0,0,0)$  is the unique saturated equilibrium of system \eqref{eq} if and only if  $b_i\le 0,\, i=1,2$.\\
(ii)   $E^{*,1}=(x_1^{*,1},u_1^{*,1},0,0)$, where 
\begin{equation}\label{E1,n=2}
x_1^{*,1}=\frac{b_1}{\la_1+a_{11}}>0, \quad u_1^{*,1}=\frac{d_1b_1}{e_1(\la_1+a_{11})}
\end{equation}
 is the unique saturated equilibrium of system \eqref{eq} if and only if 
\begin{equation}\label{E1,b1}
b_1>0 \quad \mbox{and} \quad b_1a_{21}  \geq b_2(\la_1+a_{11}).
\end{equation}
(iii)  there exists a unique positive equilibrium $E^{*,2}=(x_1^{*,2},u_1^{*,2},x_2^{*,2},u_2^{*,2})$ of system \eqref{eq}, where 
\begin{equation}
\left\{
\begin{array}{ll}
\displaystyle{x_1^{*,2}=\frac{b_1(\la_2+a_{22})-b_2a_{12}}{\det M}, \q x_2^{*,2}=\frac{ b_2(\la_1+a_{11})-b_1a_{21} }{\det M},} \\
\displaystyle{u_1^{*,2}=\frac{d_1 x_1^{*,2}}{e_1}, \quad u_2^{*,2}=\frac{d_2 x_2^{*,2}}{e_2}},
\end{array}
\right.
\end{equation}
if and only if 
\begin{equation}
b_1(\la_2+a_{22})>b_2a_{12} \quad \mbox{and} \quad b_2(\la_1+a_{11})>b_1a_{21}.
\end{equation}}
\end{exmp}

  \begin{exmp}\label{exmp2.2} {\rm Consider  system  \eqref{eq} with $n=2$, $b_1=b_2=-2,a_{12}=a_{21}=-4$, $a_{ii}=0$, $\mu_i=c_i=d_i=e_i=1$ for $i=1,2$, so that
  $M_0=M_0^-=\left[
\begin{array}{cc}
1&-4 \\
-4&1 \\
\end{array}
\right],\ 
M=\left[
\begin{array}{cc}
2&-4 \\
-4&2 \\
\end{array}
\right].$ In this situation, $M_0$ is not a non-singular M-matrix, neither $M$ is  a P-matrix, and both (0,0,0,0) and (1,1,1,1) are saturated equilibria. This simple example shows that the above setting should be used with caution.}
\end{exmp}

  %as well as the case $p=n$ (of a positive saturated equilibrium) have already been addressed in \cite{Faria:2015}, see Theorem \ref{thm2.2}. However, here we shall give sharper conditions for the global attractivity of the  positive  equilibrium. The idea now is to classify all the equilibria of \eqref{eq}, which for simplicity we assume have the form \eqref{Ep}.

%%%%%%%%%%%%%%%%%%%%%%%%%%%%%% Section 2 %%%%%%%%%%%%%%%%%%%%%%%%%%%%
\section{Main results}
\setcounter{equation}{0}

For the sake of completeness, we include here a result about  matrices which was not found in the literature.

\begin{lem}\label{lem3.1}
Let $B=[b_{ij}]$ be an $n\times n$ matrix with $b_{ij}\le 0$ for all $i\ne j$. The two  conditions are equivalent:\vskip 0cm
(i) $B$ is a non-singular M-matrix;\vskip 0cm
(ii) there exist positive vectors $\eta=(\eta_1,\dots,\eta_n), q=(q_1,\dots,q_n)$ such that
\begin{equation}\label{014}
\sum_{j=1}^n\bigl(\eta_i b_{ij}q_j+\eta_j b_{ji}q_i\bigr)>0 \q \mbox{for} \q i=1,\ldots,n.
\end{equation}
  \end{lem}

\begin{proof}
If $B$ is a non-singular M-matrix, so is its transpose $B^T$. Thus, there are positive vectors $\eta=(\eta_1,\dots,\eta_n), q=(q_1,\dots,q_n)$ such that $Bq>0$ and $B^T\eta >0$. It follows that
$$\eta_i\Big (\sum_j b_{ij}q_j\Big )>0,\q q_i\Big (\sum_j b_{ji}\eta_j\Big )>0\q {\rm for}\q i=1,\dots,n,$$
which implies \eqref{014}.

Conversely, suppose that \eqref{014} holds for some positive vectors $\eta,q\in \R^n$. 
%Define ${\cal B}=-B=[a_{ij}]$. 
Since $-b_{ij}\ge 0$ for all $i\ne j$, there exists $c>0$ such that $-B=A_0-cI$, where $A_0$ is a non-negative matrix. If $A_0=0$, then $B=cI$ and (i) is satisfied. Otherwise, by the Perron-Frobenius Theorem \cite{Berman:1979,Hofbauer:1988},
%(see e.g. [HS,~p.82, or Smith+Waltman,p.~257]), 
the spectral radius of $A_0$, $\rho:=\sup \{ |\mu|:\mu \in \sigma (A_0)\}$, is an eigenvalue of $A_0$ with a corresponding non-negative eigenvector $u$.
Clearly $\sigma (-B)=\sigma (A_0)- c$, hence $\la=\rho -c$ is the spectral bound $s(-B)$ of $-B$, and
$-Bu=\la u$. If suffices to prove that $\la<0$. In fact, since  $\sigma(-B)=-\sigma(B)$,  if $\la<0$ it follows that   all eigenvalues of $B$ have positive real parts, which shows that $B$ is a non-singular M-matrix.

For the sake of contradiction, suppose that $\la\ge 0$, and consider the ODE  system
\begin{equation}\label{015}
x_i'=-x_i [r_i+(Bx)_i],\q i=1,\dots,n,
\end{equation}
where $r$ is the  vector defined by $r=\la u\ge 0$.
We now prove that condition \eqref{014} implies that $x=0$ is a global attractor of all non-negative solutions of \eqref{015}.

By the scaling $\bar x_i=x_i/q_i\ (1\le i\le n)$, and dropping the bars for simplicity, we may assume that (ii) holds with $q_i=\cdots =q_n=1$; i.e, for some $\eta=(\eta_1,\dots,\eta_n)>0$, it holds
\begin{equation}\label{016}
\sum_{j=1}^n\bigl(\eta_i b_{ij}+\eta_j b_{ji}\bigr)>0 \q \mbox{for} \q i=1,\ldots,n.
\end{equation}
Consider the function
$V(t)=\sum_{i=1}^n\eta_i x_i(t).$
Along non-negative solutions $x(t)\not\equiv 0$ of  \eqref{015}, $V(t)\ge 0$ and
\begin{equation*}
\begin{split}
\dot V(t)&=-\sum_i\eta_i x_i(t) \Big(r_i+\sum_jb_{ij}x_j(t)\Big)\le -\sum_i\eta_i x_i(t) \Big(\sum_jb_{ij}x_j(t)\Big)\\
&\le -\sum_i\eta_i \left ( b_{ii}x_i^2(t) +\sum_{j\ne i} b_{ij} \frac{1}{2} (x_i^2(t)+x_j^2(t))\right)\\
\end{split}
\end{equation*}
because $b_{ij}\le 0$ for $i\ne j$, hence
\begin{equation*}
\begin{split}
\dot V(t)&\le -\sum_i\bigg [ \eta_i\Big (b_{ii}+\frac{1}{2}\sum_{j\ne i}b_{ij}\Big )x_i^2(t)+\eta_i\frac{1}{2}\sum_{j\ne i}b_{ij} x_j^2(t)\bigg ]\\
&=-\sum_i\bigg [ \eta_i b_{ii}x_i^2(t)+\frac{1}{2}\sum_{j\ne i}\Big (\eta_ib_{ij}+\eta_jb_{ji}\Big )x_i^2(t)\bigg ]\\
&=-\frac{1}{2}\sum_i\sum_j \bigl(\eta_i b_{ij}+\eta_j b_{ji}\bigr)x_i^2(t)<0.
\end{split}
\end{equation*}
This proves that $x=0$ is a global attractor of all non-negative solutions of \eqref{015}. But this contradicts the fact that  $x=u$ is a (non-zero) non-negative equilibrium of \eqref{015}.
\end{proof}

%\begin{rmk}\label{rmk3.1} If \eqref{cond-eta-q} holds,  for new variables $\bar x_i(t)=\frac{x_i(t)}{q_i}$ and $\bar u_i(t)=\frac{u_i(t)}{q_i}, \ 1\le i\le n$, system \eqref{eq} becomes 
%\begin{equation}\label{eq-bar}
%\hspace{-0.3cm}
%\left\{
%\begin{array}{ll}
%\displaystyle{\bar x_i^{\prime}(t)=\bar x_i(t)\biggl(b_i-\bar{\mu}_i \bar x_i(t)-\sum_{j=1}^n \bar{a}_{ij} \int_0^{\infty}K_{ij}(s)\bar x_j(t-s)\, ds} \\
%\hskip 2cm\displaystyle{-\bar c_i \int_0^{\infty}G_i(s)\bar u_i(t-s)\, ds\biggr)}, \\
%\displaystyle{\bar u_i^{\prime}(t)=-e_i \bar u_i(t)+d_i \bar x_i(t), \quad i=1,\ldots,n}, 
%\end{array}
%\right.
%\end{equation}
%where the new coefficients are given by $\bar{\mu}_i=\mu_i q_i$, $\bar{a}_{ij}=a_{ij} q_j,$ and $\bar{c}_i=c_i q_i,\ i,j=1,\ldots,n$. 
%Dropping the bars  from the new variables and coefficients of \eqref{eq-bar}, for simplicity we may consider only the case of conditions  \eqref{cond-eta-q} with $q_1=q_2=\ldots=q_n=1$. \end{rmk}

\begin{rmk}\label{rmk3.1} {\rm As in the above proof, whenever it is convenient,  one may effect the changes of variables $\bar x_i(t)=\frac{x_i(t)}{q_i}$ and $\bar u_i(t)=\frac{u_i(t)}{q_i}, \ 1\le i\le n$, which transform system \eqref{eq} into
\begin{equation}\label{eq-bar}
\hspace{-0.3cm}
\left\{
\begin{array}{ll}
\displaystyle{\bar x_i^{\prime}(t)=\bar x_i(t)\biggl(b_i-\bar{\mu}_i \bar x_i(t)-\sum_{j=1}^n \bar{a}_{ij} \int_0^{\infty}K_{ij}(s)\bar x_j(t-s)\, ds} \\
\hskip 2cm\displaystyle{-\bar c_i \int_0^{\infty}G_i(s)\bar u_i(t-s)\, ds\biggr)}, \\
\displaystyle{\bar u_i^{\prime}(t)=-e_i \bar u_i(t)+d_i \bar x_i(t), \quad i=1,\ldots,n}, 
\end{array}
\right.
\end{equation}
where the new coefficients are given by $\bar{\mu}_i=\mu_i q_i$, $\bar{a}_{ij}=a_{ij} q_j,$ and $\bar{c}_i=c_i q_i,\ i,j=1,\ldots,n$.}\end{rmk}

We are now ready to prove our main results. The Lyapunov functional used below is inspired by the ones introduced in \cite{Gopalsamy:2003, Li:2013,MTW:2015}.

\begin{thm}\label{thm3.1} Consider \eqref{eq}, assume (H0) and let
%Let $\mu_i, c_i, d_i,  e_i$ be
 %positive constants,  $b_i, a_{ij} \in \R, \ i,j \in \{1,2,\ldots,n\}$, and consider \eqref{eq}.
%$M$ is a P-matrix and let 
%$E^*=(x_1^*,u_1^*,\ldots,x_n^*,u_n^*)$
%be the unique saturated equilibrium  of \eqref{eq}.
 $E^*=(x_1^*,u_1^*,\ldots,x_p^*,u_p^*,0,0,\ldots,0,0)$ be a  saturated equilibrium  of \eqref{eq},
%Moreover, without loss of generality, suppose that  $E^*$ has the form $(x_1^*,u_1^*,\ldots,x_p^*,u_p^*,0,0,\ldots,0,0)$
%, with $x_i^*>0$ for $i=1,\dots,p$, 
 for some $p\in\{0,1,\dots,n\}$, where   $E^*= 0$ if $p=0$.  If the matrix
\begin{equation}\label{cal M0}\hat{\cal M}_{0,p}:=[\de_{ij}\mu_i-|\tilde{a}_{ij}|],
\end{equation}
where
\begin{equation}\label{modified_a}
\tilde{a}_{ij}:=a_{ij}^-=\max \{ 0, -a_{ij}\}\ \mbox{for} \ i,j=p+1,\dots, n, \ \mbox{and } \ \tilde{a}_{ij}=a_{ij} \  \ \mbox{otherwise},
\end{equation}
is a non-singular M-matrix,
then $E^*$ is globally attractive. 
\end{thm}

\begin{proof}  % By Theorem \ref{thm2.1},  there exists a unique saturated equilibrium $E^*$  of \eqref{eq}. 
With $p=0$, we have the case $E^*=0$, addressed in Theorem \ref{thm2.2}(i). Note that in this situation $\hat{\cal M}_{0,p}=M_0^-$. So let $E^*\ne 0$ be  a saturated equilibrium.   

After reordering the variables, we may suppose that this equilibrium  $E^*$ has the form $E^*=E^{*,k}=(x_1^*,u_1^*,\ldots,x_k^*,u_k^*,0,0,\ldots,0,0)$ given by \eqref{Ep}, with $x_i^*>0$ for $i=1,\dots,k$, for some $k \in \{1,2,\ldots,p\}$. If  $\hat{\cal M}_{0,p}$ is a nonsingular M-matrix and $1\le k<p$, then $\hat{\cal M}_{0,k}$ is a nonsingular M-matrix as well, so we may suppose that $p=k$; otherwise, we replace $p$ by $k$ in the computations below. We now prove that this $E^*=E^{*,p}$ is globally attractive.

From Lemma  \ref{lem3.1},   take positive vectors  $\eta=(\eta_1,\eta_2,\ldots,\eta_n)$ and $q=(q_1,\dots,q_n)$ for which 
 \begin{equation}\label{cond-eta-q}
\displaystyle{ \eta_i \mu_iq_i>\sum_{j =1}^n \frac{1}{2}\biggl(\eta_i |\tilde{a}_{ij}|q_j+\eta_j |\tilde{a}_{ji}|q_i\biggr),\q i=1,\dots,n,}
\end{equation}
for $\tilde{a}_{ij}$ given by \eqref{modified_a}.
Effecting a scaling of the variables as in Remark \ref{rmk3.1},  
and dropping the bars  from the new variables and coefficients in \eqref{eq-bar}  for simplicity,  we assume  \eqref{cond-eta-q}  with $q_1=q_2=\ldots=q_n=1$, i.e.,
%we may take a positive vector  $\eta=(\eta_1,\eta_2,\ldots,\eta_n)$ such that   
\begin{equation}\label{eq-eta_tilde}
\displaystyle{ \eta_i \mu_i>\sum_{j =1}^n \frac{1}{2}\biggl(\eta_i |\tilde{a}_{ij}|+\eta_j |\tilde{a}_{ji}|\biggr),\q i=1,\ldots,n}.
\end{equation}

%or equivalently,
%\begin{equation}\label{eq-eta_tilde}
%\begin{split}
%\eta_i \mu_i&> \frac{1}{2}\sum_{j =1}^n\biggl(\eta_i |a_{ij}|+\eta_j |a_{ji}|\biggr), \q i=1,\ldots,p,\\
% \eta_i \mu_i&> \frac{1}{2}\left [\sum_{j =1}^p\biggl(\eta_i |a_{ij}|+\eta_j |a_{ji}|\biggr)
% +\sum_{j =p+1}^n\biggl(\eta_i a_{ij}^-+\eta_j a_{ji}^-\biggr)\right], \q i=p+1,\ldots,n.
%\end{split}
%\end{equation}
 For $x>0$ and $x^*\ge 0$,
define the functions $g(x)=x-1-\ln x$ and 
$
G(x,x^*)=\left\{
\begin{array}{ll}
x^*g\bigl(\frac{x}{x^*}\bigr), \ & \mbox{if} \ x^*>0, \\ 
x, \ & \mbox{if} \ x^*=0.
\end{array}\right.
$ Thus $G(x,x_i^*)=x_i^*g\bigl(\frac{x}{x_i^*}\bigr)$ if $i=1,\ldots,p$ and $G(x,x_i^*)=x$ if $i=p+1,\ldots,n$.
%\begin{equation*}
%\left.
%\begin{array}{ll}
%G(x,x_i^*)=\left\{
%\begin{array}{ll}
%x_i^*g\bigl(\frac{x}{x_i^*}\bigr), \ & \mbox{if} \ i=1,\ldots,p \\ 
%x, \ & \mbox{if} \ i=p+1,\ldots,n.
%\end{array}
%\right.
%\end{array}
%\right.
%\end{equation*}
Define
\begin{equation}\label{U1}
U_1(t)=\sum_{i=1}^n \eta_i \biggl(G(x_i(t),x_i^*)+\frac{c_i}{d_i}\frac{(u_i(t)-u_i^*)^2}{2}\biggr),
\end{equation}
Along positive solutions $X(t)=(x_1(t),u_1(t),x_2(t),u_2(t),\ldots,x_n(t),u_n(t))$  of  system \eqref{eq}, we have   ${\dot{U}_1(t)=\sum_{i=1}^n \eta_i \biggl\{\frac{d}{dt}G(x_i(t),x_i^*)+\frac{c_i}{d_i}(u_i(t)-u_i^*) u_i^{\prime}(t)\biggr\}}$ and
%\begin{equation}\label{DU1}
%\left.
%\begin{array}{ll}
%\displaystyle{\dot{U}_1(t)=\sum_{i=1}^n \eta_i \biggl\{\frac{d}{dt}G(x_i(t),x_i^*)+\frac{c_i}{d_i}(u_i(t)-u_i^*) u_i^{\prime}(t)\biggr\}}.
%\end{array}
%\right.
%\end{equation}
%Note that 
%
%$$\frac{d}{dt}G(x_i(t),x_i^*)=
%\left\{
%\begin{array}{ll}
%x_i^*\frac{d}{dt}g\bigl(\frac{x_i(t)}{x_i^*}\bigr),  & \mbox{if} \ i=1,\ldots,p \\
%x_i^{\prime}(t),  & \mbox{if}\ i=p+1,\ldots,n
%\end{array}
%\right. .$$
\begin{equation}\label{3.15}
\hspace{-0.5cm}
\displaystyle{\frac{d}{dt}G(x_i(t),x_i^*)}=
\left\{
\begin{array}{ll}
(x_i(t)-x_i^*)\frac{x_i^{\prime}(t)}{x_i(t)}, \ & \mbox{if} \ i=1,\ldots,p \\
x_i^{\prime}(t), \ & \mbox{if}\ i=p+1,\ldots,n. 
\end{array}
\right.
\end{equation}
For the case $x_i^*>0$ (i.e., for $ i=1,\ldots,p$), 
%\begin{equation}\label{3.15}
%\displaystyle{x_i^*\frac{d}{dt}g\biggl(\frac{x_i(t)}{x_i^*}\biggr)=(x_i(t)-x_i^*)\frac{x_i^{\prime}(t)}{x_i(t)}}.
%\end{equation}
since $\displaystyle{b_i=\mu_i x_i^*+\sum_{j=1}^n a_{ij} x_j^*+c_i u_i^*}$ and $e_i u_i^*=d_i x_i^*$, we have 
\begin{equation}\label{3.16}
%\hspace{-0.2cm}
%\left.
%\begin{array}{ll}
\begin{split}
x_i^{\prime}(t)
&=x_i(t)\biggl(-\mu_i (x_i(t)-x_i^*)-\sum_{j=1}^{p} a_{ij}  \int_0^{\infty}K_{ij}(s)(x_j(t-s)-x_j^*)\, ds \\
&-c_i \int_0^{\infty}G_{i}(s)(u_i(t-s)-u_i^*)\, ds-\sum_{j=p+1}^n a_{ij} \int_0^{\infty}K_{ij}(s)x_j(t-s)\, ds\biggr), \\
u_i^{\prime}(t)&=-e_i (u_i(t)-u_i^*)+d_i (x_i(t)-x_i^*), \quad i=1,2,\ldots,p. 
\end{split}
%\end{array}
%\right.
\end{equation}
For the case $x_i^*=0$ (i.e., for $ i=p+1,\ldots,n)$, from  \eqref{2.1} we have $\displaystyle{b_i \leq \sum_{j=1}^{p} a_{ij} x_j^*}$, and similarly 
\begin{equation}\label{3.17}
%\left.
%\begin{array}{ll}
\begin{split}
x_i^{\prime}(t) &\leq x_i(t) \biggl(-\mu_i x_i(t)-\sum_{j=1}^{p} a_{ij} \int_0^{\infty}K_{ij}(s)(x_j(t-s)-x_j^*)\, ds \\
& -c_i \int_0^{\infty}G_{i}(s)u_i(t-s)\, ds-\sum_{j=p+1}^n a_{ij} \int_0^{\infty}K_{ij}(s)x_j(t-s)\, ds \biggr). 
\end{split}
%\end{array}
%\right.
\end{equation}
Summing up  \eqref{3.15},  \eqref{3.16}, \eqref{3.17}, we obtain
\begin{equation*}
\begin{split}
\dot{U}_1(t)&\leq \sum_{i=1}^{p} \eta_i \biggl\{(x_i(t)-x_i^*)\biggl(-\mu_i (x_i(t)-x_i^*)-\sum_{j=1}^{p} a_{ij} \int_0^{\infty}K_{ij}(s)(x_j(t-s)-x_j^*)\, ds \\
&-c_i \int_0^{\infty}G_{i}(s)(u_i(t-s)-u_i^*)\, ds-\sum_{j=p+1}^n a_{ij} \int_0^{\infty}K_{ij}(s)x_j(t-s)\, ds\biggr) \\
&+\frac{c_i}{d_i} (u_i(t)-u_i^*)\biggl(-e_i (u_i(t)-u_i^*)+d_i (x_i(t)-x_i^*)\biggr)\biggr\}\\
 &+\sum_{i=p+1}^n \eta_i \biggl\{x_i(t)\biggl(-\mu_i x_i(t)-\sum_{j=p+1}^n a_{ij} \int_0^{\infty}K_{ij}(s)x_j(t-s)\, ds\\
&-c_i \int_0^{\infty}G_{i}(s)u_i(t-s)\, ds-\sum_{j=1}^{p} a_{ij} \int_0^{\infty}K_{ij}(s)(x_j(t-s)-x_j^*)\, ds\biggr) \\
 &+\frac{c_i}{d_i} u_i(t)\biggl(-e_i u_i(t)+d_i x_i(t)\biggr)\biggr\},
\end{split}
\end{equation*}
thus
\begin{equation}\label{U-1}
\left.
\begin{array}{ll}
\displaystyle{\dot{U}_1(t)\le \sum_{i=1}^{p} \eta_i \biggl(-\mu_i(x_i(t)-x_i^*)^2-\frac{c_i e_i}{d_i}(u_i(t)-u_i^*)^2} \\
\qq \displaystyle{-\sum_{j=1}^{p} a_{ij} \int_0^{\infty}K_{ij}(s)(x_i(t)-x_i^*)(x_j(t-s)-x_j^*)\, ds} \\
\qq \displaystyle{-c_i \int_0^{\infty}G_{i}(s)(x_i(t)-x_i^*)(u_i(t-s)-u_i(t))\, ds\biggr)} \\
\qq \displaystyle{-\sum_{i=1}^p \eta_i \sum_{j=p+1}^n a_{ij} \int_0^{\infty}K_{ij}(s)(x_i(t)-x_i^*)x_j(t-s)\, ds} \\ 
\qq \displaystyle{+\sum_{i=p+1}^n \eta_i \biggl(-\mu_i x_i^2(t)-\frac{c_i e_i}{d_i}u_i^2(t)-\sum_{j=p+1}^n a_{ij} \int_0^{\infty}K_{ij}(s)x_i(t)x_j(t-s)\, ds} \\
\qq \displaystyle{-c_i \int_0^{\infty}G_{i}(s)x_i(t)(u_i(t-s)-u_i(t))\, ds \biggr)} \\ 
\qq\displaystyle{-\sum_{i=p+1}^n \eta_i \sum_{j=1}^p a_{ij} \int_0^{\infty}K_{ij}(s)x_i(t) (x_j(t-s)-x_j^*)\, ds}. \\
\end{array}
\right.
\end{equation}
We now use $-a_{ij}\le |a_{ij}|=|\tilde{a}_{ij}|$ for either $i\le p$ or $j\le p$, and $-a_{ij}\le a_{ij}^{-}=|\tilde{a}_{ij}| $ for $i,j>p$, and, for $ i,j=1,\dots,n,$ the estimates
\[
\begin{array}{ll}
\displaystyle{|(x_i(t)-x_i^*)(x_j(t-s)-x_j^*)| \leq \frac{1}{2} \biggl((x_i(t)-x_i^*)^2+(x_j(t-s)-x_j^*)^2\biggr)}, \\
 \displaystyle {|(x_i(t)-x_i)(u_i(t-s)-u_i(t))| \leq \frac{1}{2} \Bigl(\vare (x_i(t)-x_i^*)^2+\frac{1}{\vare}(u_i(t-s)-u_i(t))^2 \Bigr)},
\end{array}
\]
where $\vare$ is a positive constant.
We therefore derive
\begin{equation}\label{DU-1}
\hspace{-0.3cm}
\begin{split}
\dot{U}_1(t)&\leq \sum_{i=1}^n \eta_i \biggl\{\biggl(-\mu_i (x_i(t)-x_i^*)^2-\frac{c_i e_i}{d_i}(u_i(t)-u_i^*)^2\biggr)\\
%\ \ \ \displaystyle{+\frac{|a_{ii}|}{2} \biggl( (x_i(t)-x_i^*)^2+\int_0^{\infty}K_{ii}(s)(x_i(t-s)-x_i^*)^2 ds\biggr)} \\ 
&+\sum_{j=1}^n \frac{|\tilde{a}_{ij}|}{2}\biggl( (x_i(t)-x_i^*)^2+\int_0^{\infty}K_{ij}(s)(x_j(t-s)-x_j^*)^2 ds\biggr) \\
&+\frac{c_i}{2} \biggl( \vare (x_i(t)-x_i^*)^2+\frac{1}{\vare} \int_0^{\infty}G_{i}(s)(u_i(t-s)-u_i(t))^2 ds\biggr) \biggr\}.
\end{split}
\end{equation}
By our assumption  \eqref{eq-eta_tilde},
we may choose  $\vare>0$ sufficiently small such that 
\begin{equation}\label{cond-varepsilon}
\displaystyle{ \eta_i (\mu_i-|\tilde{a}_{ii}|)>\sum_{j \neq i} \frac{1}{2}\biggl(\eta_i |\tilde{a}_{ij}|+\eta_j |\tilde{a}_{ji}|\biggr)+\frac{c_i \vare}{2}, \ \mbox{for} \ i=1,2,\ldots,n}.
\end{equation}
Set 
\begin{equation}\label{U2}
\begin{split}
U_2(t)&=\sum_{i=1}^n \eta_i \biggl(\sum_{j=1}^n \frac{|\tilde{a}_{ij}|}{2} \int_0^{\infty}K_{ij}(s) \int_{t-s}^t (x_j(u)-x_j^*)^2 du ds \\
&+\frac{c_i}{2 \vare} \int_0^{\infty}G_{i}(s) \int_{t-s}^t (u_i(u)-u_i(t))^2 du ds \biggr).
\end{split}
\end{equation}
Note that, since $\hat{\cal M}_{0,p}$ is a non-singular M-matrix and $M_0^-\ge \hat{\cal M}_{0,p}$, we know that $M_0^-$ is also a non-singular M-matrix. From Theorem \ref{thm2.1}, it follows that all solutions are bounded. Boundedness of solutions together with condition \eqref{1.3} imply that  $U_2(t)$ is well defined. 

Clearly $U_2(t)\ge 0$ for $t\ge 0$. Calculating the derivative of $U_2$ along solutions, we obtain 
\begin{equation}\label{U-2}
\hspace{-0.5cm}
\begin{split}
\dot{U}_2(t)&
=\sum_{i=1}^n \eta_i \biggl\{\sum_{j=1}^n \frac{|\tilde{a}_{ij}|}{2} \biggl((x_j(t)-x_j^*)^2-\int_0^{\infty}K_{ij}(s)(x_j(t-s)-x_j^*)^2 ds\biggr) \\
&-\frac{c_i}{2 \vare}\int_0^{\infty}G_{i}(s)(u_i(t-s)-u_i(t))^2 ds
\biggr\}.
\end{split}
\end{equation}
Define $U(t)=U_1(t)+U_2(t).$
It follows from \eqref{DU-1},  \eqref{cond-varepsilon} and \eqref{U-2}  that 
\begin{equation}\label{U}
\hspace{-0.3cm}
\left.
\begin{array}{ll}
 \displaystyle{\dot{U}(t)\leq \sum_{i=1}^n \eta_i \biggl\{\biggl(-(\mu_i-|\tilde{a}_{ii}|)(x_i(t)-x_i^*)^2-\frac{c_i e_i}{d_i}(u_i(t)-u_i^*)^2\biggr)} \\
\qq \displaystyle{+\sum_{j \neq i} \frac{|\tilde{a}_{ij}|}{2} \biggl((x_i(t)-x_i^*)^2 +(x_j(t)-x_j^*)^2\biggr)+\frac{c_i}{2} \vare (x_i(t)-x_i^*)^2 \biggr\}} \\
\qq \displaystyle{= -\sum_{i=1}^n \biggl( \eta_i (\mu_i-|\tilde{a}_{ii}|)-\sum_{j \neq i} \frac{1}{2}( \eta_i |\tilde{a}_{ij}|+\eta_j |\tilde{a}_{ji}|)-\frac{c_i \vare}{2} \biggr)(x_i(t)-x_i^*)^2} \\
\qq \displaystyle{-\sum_{i=1}^n \eta_i \frac{c_i e_i}{d_i}(u_i(t)-u_i^*)^2 \leq 0},
\end{array}
\right.
\end{equation}
where the inequality is strict if $x(t)\ne x^*$.
Since all coordinates $x_i(t)$ are bounded and have bounded derivatives,  $(x_i(t)-x_i^*)^2$ are uniformly continuous on $[0,\infty)$. From \eqref{U} and \eqref{cond-varepsilon}, we may write
$$\dot U(t)\le -\sum_{i=1}^n \left[A_i (x_i(t)-x_i^*)^2+B_i(u_i(t)-u_i^*)^2\right],$$
where $A_i,B_i>0, i=1,\dots,n$. Integrating the above inequality on $[0,t]$ for  $t>0$, for any $i\in \{1,\dots,n\}$ we obtain
$$A_i\int_0^t(x_i(s)-x_i^*)^2\, ds\le U(0),\q B_i\int_0^t(u_i(s)-u_i^*)^2\, ds\le U(0),$$
 hence $(x_i(t)-x_i^*)^2,(u_i(t)-u_i^*)^2$ are integrable on $[0,\infty)$. By Barbalat's Lemma (see e.g. \cite[p.~4]{Gopal}), it follows that
$$\lim_{t\to\infty} (x_i(t)-x_i^*)^2=0,\ \lim_{t\to\infty} (u_i(t)-u_i^*)^2=0.$$
The proof is complete. 
\end{proof}  

 Sufficient conditions for the existence of a positive equilibrium and its global attractivity are as follows:
\begin{thm}\label{thm3.2} Assume (H0), suppose that
the matrix $ \hat{M}_0=[\de_{ij}\mu_i-|a_{ij}|]$  is a non-singular M-matrix and that $R_i^n>0, i=1,\dots,n$, where $R_i^n$ are given by \eqref{Ri}.
 Then, there exists a unique positive equilibrium $E^*=(x_1^*,u_1^*,\ldots,x_n^*,u_n^*)$ of \eqref{eq}, which is globally attractive. 
\end{thm}

\begin{proof}
Recall the matrices $M_0,M$ and $\hat{M}_0$ defined in \eqref{1.5} and \eqref{2.3}. It is enough to prove that in this situation $M$ is a P-matrix. In fact, if $M$ is a P-matrix, Theorem \ref{thm2.1} implies that there exists a unique saturated equilibrium $E^*$  of \eqref{eq}, in which  case, as $\hat{\cal M}_{0,n}= \hat{M}_0$, the result  is a consequence of Theorem \ref{thm3.1}.

Note that $\check{M}:=[\la_i\de_{ij}-|a_{ij}|]=[(\mu_i+c_i\frac{d_i}{e_i})\de_{ij}-|a_{ij}|]\ge \hat M_0$. If $\hat{M}_0$  is a non-singular M-matrix, then   $\check{M}$ is a non-singular M-matrix as well, thus there exists a positive vector $v$ such that  $\check{M}v>0$. In particular, it follows that
$-M$ is LV-stable, hence  $M$ is a P-matrix. See \cite[pp.~201-202]{Hofbauer:1988}  and \cite{Faria:2015} for definitions and details. 
    \end{proof}

\begin{rmk}\label{rmk3.2} {\rm It is opportune to mention that the above proof of Theorem \ref{thm3.1} is applicable to non-negative solutions of differential inequalities
\begin{equation}\label{eqIneq}
\left\{
\begin{array}{ll}
\displaystyle{x_i^{\prime}(t)\le x_i(t)\biggl(b_i-\mu_i x_i(t)-\sum_{j=1}^n a_{ij} \int_0^{\infty}K_{ij}(s)x_j(t-s)\, ds} \\
\ \ \ \ \ \ \ \ \ \ \displaystyle{-c_i \int_0^{\infty}G_i(s)u_i(t-s)\, ds\biggr)}, \\
\displaystyle{u_i^{\prime}(t)\le -e_i u_i(t)+d_i x_i(t), \quad i=1,\ldots,n}, \\
\end{array}
\right.
\end{equation}
for coefficients and kernels satisfying (H0), and $E^*$ still a saturated equilibrium of \eqref{eq}. This observation also permits us to apply the above result to {\it  non-autonomous} Lotka-Volterra models with infinite delays and feedback controls  of the form
\begin{equation}\label{eqNonA}
\left\{
\begin{array}{ll}
\displaystyle{x_i^{\prime}(t)=x_i(t)\biggl(\be_i(t)- m_i (t)x_i(t)-\sum_{j=1}^n \al_{ij}(t) \int_0^{\infty}K_{ij}(s)x_j(t-s)\, ds} \\
\ \ \ \ \ \ \ \ \ \ \displaystyle{-k_i (t)\int_0^{\infty}G_i(s)u_i(t-s)\, ds\biggr)}, \\
\displaystyle{u_i^{\prime}(t)= -\epsilon_i (t)u_i(t)+\de_i (t)x_i(t), \quad i=1,\ldots,n}, \\
\end{array}
\right.
\end{equation}
where the non-negative kernels $K_{ij}(t),G_i(t)$ satisfy \eqref{1.2}, \eqref{1.3},   $\be_i(t), m_i(t),\al_{ij}(t), k_i(t), \epsilon_i (t),\de_i (t)$ are continuous  and bounded in $[0,\infty)$, with $m_i(t),  \epsilon_i (t)$ bounded below by  positive constants and $k_i(t),\de_i (t)\ge 0$. In fact, non-negative solutions of the non-autonomous system \eqref{eqNonA} satisfy \eqref{eqIneq}, with 
\begin{equation*}
\begin{array}{ll}
&b_i=\sup_{t\ge 0} \be_i(t),\ d_i=\sup_{t\ge 0} \de_i(t),\\
& \mu_i=\inf_{t\ge 0} m_i(t),\ a_{ij}=\inf_{t\ge 0}\al_{ij}(t),\ c_i=\inf_{t\ge 0}k_i(t),\ e_i=\inf_{t\ge 0} \epsilon_i (t),\q i,j=1,\dots,n.
\end{array}
\end{equation*}}
\end{rmk}

\begin{rmk}\label{rmk3.3} {\rm
 %, where the more restrictive condition  of $\hat{M}=[\delta_{ij} \mu_i-|a_{ij}|-\delta_{ij}\frac{c_i d_i}{e_i}]_{n \times n}$ being an M-matrix was imposed.  %Moreover, applying the  idea in \cite[Theorem 4.2]{Faria:2015}, the condition \eqref{eq-eta-q_tilde} is weaker than \eqref{eq-eta-q}. The global attractivity of 
 As already noticed by the authors in \cite{Faria:2015}, in fading memory spaces  (see \cite{HMN} for a definition)  the saturated equilibrium $E^*$ of \eqref{eq} is asymptotically stable if it is globally attractive and positive, however $E^*$ is not necessarily asymptotically stable if it lies on the boundary of the positive cone (although its linearization
  is stable).}\end{rmk}

  Theorem \ref{thm3.1} states  the global attractivity of the unique saturated equilibrium of \eqref{eq}, if  $\hat{\cal M}_{0,p}$ defined by \eqref{cal M0} is a non-singular M-matrix (or equivalently, if
 \eqref{cond-eta-q} holds, for some positive vectors $\eta,q$).
% which without loss of generality we may suppose has the form $(x_1^*,u_1^*,\ldots,x_p^*,u_p^*,0,0,\ldots,0,0)$ . 
 This sufficient condition {\it does not depend on the feedback controls}, and strongly improves results in the literature \cite{Faria:2015, Gopalsamy:2003,Li:2013, Shi:2012}. E.g.,
 applying Theorem \ref{thm3.1} to a planar Lotka-Volterra system, we obtain the following corollary (compare with \cite[Theorems 1.1 and 1.2]{Li:2013} and  \cite[Proposition 5.3]{Faria:2015}):

\begin{cor}\label{cor3.1} For \eqref{eq} with $n=2$, assume \eqref{M0-matrix-2} and as before set
$a_{ii}^-=\max \{ 0,-a_{ii}\}, i=1,2$. With the notation in Example \ref{exmp2.1},  \vskip 0cm
i) if $E^{*,1}=(x_1^{*,1},u_1^{*,1},0,0)$ is the  unique saturated equilibrium and 
$\hat{\cal M}_{0,1}=\left[
\begin{array}{cc}
\mu_1-|{a}_{11}|&-|a_{12}| \\
-|a_{21}|&\mu_2-a_{22}^-\\
\end{array}
\right]$ is a non-singular M-matrix, i.e.,
\begin{equation}\label{hatM0-matrix-2-1}
\mu_1-|a_{11}|>0, \ \mu_2-a_{22}^->0\q and \q (\mu_1-|a_{11}|)(\mu_2-a_{22}^-)>|a_{12} a_{21}|,
\end{equation}
then $E^{*,1}$ is globally attractive; \vskip 0cm
ii) if there exists a positive equilibrium $E^{*,2}$ and
the $2 \times 2$ matrix $\hat{M}_0=[\delta_{ij} \mu_i-|a_{ij}|]$ is a nonsingular M-matrix, that is, 
\begin{equation}\label{hatM0-matrix-2-3}
\mu_i-|a_{ii}|>0, \ i=1,2, \quad \mbox{and} \quad (\mu_1-|a_{11}|)(\mu_2-|a_{22}|)>|a_{12}a_{21}|,
\end{equation}
then $E^{*,2}$  is globally attractive. 
\end{cor}

%%%%%%%%%%%%%%

\section{Sharper results on partial extinction}
\setcounter{equation}{0}

 If the saturated
 equilibrium is on the boundary of the non-negative cone $\R^{2n}_+$, Theorem \ref{thm3.1} establishes sufficient conditions for some of the populations $x_j(t)$ to be driven to extinction. 
For other results on partial extinction of Lotka-Volterra systems with infinite delay (and with or without controls), see e.g. \cite{Faria:2010,Li:2013,MOca:12, MOca:15,MOca:06,Muroya:2014,Shi:2012}.
In this section, we study  criteria to have this
partial extinction associated with the sufficient conditions  in Theorem \ref{thm3.2}  for the global attractivity, not of the saturated equilibrium of the original system \eqref{eq} but instead of the positive equilibrium of the reduced and rearranged system \eqref{reduced-eq}. In other words, and with the notation in Section 2, we address the  question: when does  the  global attractivity of the positive equilibrium $\tilde E^{*,p}$ of the reduced and rearranged system \eqref{reduced-eq}   imply the global attractivity of the equilibrium $E^{*,p}$ of \eqref{eq}?

%saturated equilibrium  $(x_1^*,u_1^*,\ldots,x_p^*,u_p^*,0,0,\ldots,0,0)$ on the boundary of the positive cone $[0,\infty)^n$.

 \begin{thm}\label{thm4.1} For system \eqref{eq} under (H0), assume that  for some $p\in \{1,\dots, n-1\}$ the following holds: \vskip 0cm
 %, where
  %$\mu_i,  c_i  d_i,  e_i$ are positive constants, $b_i, a_{ij} \in \R,$ and $K_{ij}(s),G_i(s)$ satisfy \eqref{1.2} %and satisfy \eqref{1.3}. Assume that:\\
 (i)  $E^{*,p}=(x_1^*,u_1^*,\dots, x_p^*,u_p^*,0,0\ldots,0,0)$ is a saturated equilibrium of \eqref{eq};\vskip 0cm
 (ii) for any positive solution $(x_1(t),u_1(t),\dots,x_n(t),u_n(t))$ of \eqref{eq}, $$x_q(t)\in L^2[0,\infty)\q {\rm and}\q
 \lim_{t \to \infty}x_q(t)=0,\q {\rm for}\  q=p+1,\dots,n;$$\vskip 0cm
  (iii)  
% for any positive solution $(x_1(t),u_1(t),\dots,x_n(t),u_n(t))$ of \eqref{eq}, $x_q(t)\in L^2[0,\infty)$ for $ q=p+1,\dots,n$;\\
%%\lim_{t \to \infty}\int_0^{\infty}K_{iq}(s)x_q(t-s)\, ds=0,\q  \mbox{for any} \ i \in \{1,\ldots,p\}
%(iv) 
%For $M^{(p)}_0:=[\de_{ij}\mu_i+a_{ij}]\ (i,j=1,\dots,p)$ the uncontrolled community matrix of 
for the ``reduced and rearranged system" \eqref{reduced-eq},
the matrix $\hat{M}^{(p)}_0:=[\de_{ij}\mu_i-|a_{ij}|]\ (i,j=1,\dots,p)$ is a non-singular M-matrix. \vskip 0cm
 Then this saturated equilibrium $E^{*,p}$ is a global attractor for \eqref{eq}. 
\end{thm}

First, we give some auxiliary results.

\begin{lem}\label{lem4.1}  Let $K_{ij}(s)$ satisfy \eqref{1.2}, \eqref{1.3} and $x_j(t)$ be a component of a solution of \eqref{eq}-\eqref{initial}. \vskip 0cm
(i) If $\dps\lim_{t\to \infty} x_j(t)=0$ and $\al>0$, then $\dps\lim_{t \to \infty}\int_0^{\infty}K_{ij}(s)x_j^\al(t-s)\, ds=0.$\vskip 0cm
(ii) If $x_j\in L^2[0,\infty)$, then the function 
$\int_0^\infty K_{ij}(s)x_j^2(t-s)\, ds, \ t\ge 0,$
is in $L^1[0,\infty)$.
 \end{lem}

\begin{proof} Let $X(t)=(x_1(t),u_1(t),\dots,x_n(t),u_n(t))$ be a solution  of \eqref{eq}-\eqref{initial}.

(i) Let $M>0$ be such that $\sup_{t\in\R}|x_j^\al(t) |\le M$ and fix any $i\in\{1,\dots,n\}$. For any  $\vare>0$, there are $T_0, T_1>0$ such that
$\int_{T_0}^\infty K_{ij}(s)\, ds<\vare/(2M)$ and $0\le x_j^\al(t)<\vare /2$ for $t\ge T_1$. Thus, for $t\ge T_0+T_1$ one obtains
\begin{equation*}
\begin{split}\int_0^\infty K_{ij}(s)x_j^\al(t-s)\, ds &=\int_0^{T_0} K_{ij}(s)x_j^\al(t-s)\, ds +\int_{T_0}^\infty K_{ij}(s)x_j^\al(t-s)\, ds \\
&\le \frac{\vare}{2}\int_0^{T_0} K_{ij}(s)\, ds+M\int_{T_0}^\infty K_{ij}(s)\, ds <\vare.
\end{split}
\end{equation*}

(ii)  Fix $b>0$, and write
$$\int_0^b dt\int_0^\infty K_{ij}(s)x_j^2(t-s)\, ds=\int_0^\infty K_{ij}(s)\left (\int_{-s}^{b-s}x_j^2(y)\, dy\right) ds.$$
Let $M>0$ be such that $\sup_{t\le 0}x^2_j(t)\le M$ and consider any $s\ge 0$. If $b-s\le 0$, then 
$\int_{-s}^{b-s}x_j^2(y)\, dy\le \int_{-s}^0x_j^2(y)\, dy\le Ms$.  If $b-s> 0$, then 
$$\int_{-s}^{b-s}x_j^2(y)\, dy=\int_{-s}^0x_j^2(y)\, dy+\int_{0}^{b-s}x_j^2(y)\, dy\le sM+\|x_j\|_{L^2[0,\infty)}^2.$$
Hence, using \eqref{1.3}, for any $b>0$ we have
$$\int_0^b dt\int_0^\infty K_{ij}(s)x_j^2(t-s)\, ds\le M\int_0^\infty sK_{ij}(s)\, ds+\|x_j\|_{L^2[0,\infty)}^2<\infty,$$ which proves that the function $t\mapsto \int_0^\infty K_{ij}(s)x_j^2(t-s)\, ds$ is integrable on $[0,\infty)$.  
\end{proof}

\begin{lem}\label{lem4.2} Consider a solution $X(t)=(x_1(t),u_1(t),x_2(t),u_2(t),\ldots,x_n(t),u_n(t))$  of \eqref{eq} with  initial conditions  \eqref{initial}. Then,
% If $x_i(t)$ is bounded on $[0,\infty)$, then $u_i(t)$ is bounded as well with
$$0\le \liminf_{t\to \infty} x_i(t)\le  \frac{e_i}{d_i} \liminf_{t\to \infty} u_i(t)\le \frac{e_i}{d_i}\limsup_{t\to \infty} u_i(t)\le \limsup_{t\to \infty} x_i(t),\q i=1,\dots,n.$$
%If in addition  $M_0^-=[\de_{ij}\mu_i-a_{ij}^-]$
%is a non-singular M-matrix, there exists $C>0$ such that $X(t)\in [0,C]^{2n}$ for $t\ge 0$.
\end{lem}

\begin{proof} Integration of
$u_i'(t)=-e_i u_i(t)+d_i x_i(t)$ gives
$u_i(t)=u_i(0)e^{-e_it}+d_ie^{-e_it}\int_0^te^{e_is}x_i(s)\, ds$ for $ t\ge 0,$ which leads to the above estimates. 
\end{proof}

\begin{lem}\label{lem4.3} Under   the hypotheses  (ii) and (iii) of Theorem \ref{thm4.1} (except the requirement that $x_q(t)\in L^2[0,\infty)$ for $q=p+1,\dots,n$), all  solutions of \eqref{eq} with initial conditions \eqref{initial} are bounded. \end{lem}

\begin{proof} It was already observed that solutions  $X(t)$ of \eqref{eq} with  initial conditions  \eqref{initial} are defined and positive for $t\ge 0$. Fix a solution $X(t)=(x_1(t),u_1(t),x_2(t),u_2(t),\ldots,x_n(t),u_n(t))$.

%Fix an initial condition \eqref{initial} and let  $X(t)=(x_1(t),u_1(t),x_2(t),u_2(t),\ldots,x_n(t),u_n(t))$  be the  solution of  \eqref{eq}-\eqref{initial}. 

Write the uncontrolled community matrix $M_0$ as
\begin{equation}\label{M0,11}
M_0=
\left [
\begin{array}{cc}
M_{0,11}& A_{12} \\
A_{21}& {M}_{0,22}\\
\end{array}
\right], 
\end{equation}
where $M_{0,11}:=[\de_{ij}\mu_i+a_{ij}]\ (i,j=1,\dots,p)$ is the $p\times p$ uncontrolled community matrix 
for the reduced  system \eqref{reduced-eq}
 and ${M}_{0,22}$ is an $(n-p)\times (n-p)$ matrix.
 By hypothesis (iii),  $\hat M_{0,11}:=\hat M_0^{(p)}=\big [\de_{ij}\mu_i-|a_{ij}|\big ]_{p\times p}$ is a non-singular M-matrix, hence there is a positive vector $\eta\in\R^p$ such that $\hat M_{0,11}\eta >0$. After a scaling, take $\eta=(1,\dots,1)$. Hence, one can choose $\de>0$ small enough so that
% $$\mu_i-\sum_{j=1}^p |a_{ij}|-2\de >0,\q i=1,\dots, p.$$
 $$\mu_i-\sum_{j=1}^p |a_{ij}|-\de >0,\q i=1,\dots, p. $$
 Define  $$h_i(t):=\sum_{j=p+1}^n a_{ij}\int_0^\infty K_{ij}(s)x_j(t-s)\, ds,\q t\ge 0.$$
 By assumption (ii) and Lemma \ref{lem4.1}, choose  $T_0>0$ large such that $x_j(t)\le \de$ for $t\ge T_0,j=p+1,\dots,n$ and $|h_i(t)|\le \de $ for $t\ge T_0, i=1,\dots, p$.

Consider $\R^p$ endowed with the maximum norm $|\cdot|$.
 We claim that
 $\sup_{t\ge 0}|(x_1(t),\dots,x_p(t))|<\infty.$
 
 Otherwise, for any $K>\max \{1+\frac{b_1}{\de},\dots, 1+\frac{b_p}{\de},\sup_{t\le T_0}|(x_1(t),\dots,x_p(t))|\}$, there is $T=T(K)>T_0$ such that
 $$|(x_1(T),\dots,x_p(T))|\ge K\q {\rm and}\q |(x_1(t),\dots,x_p(t))|\le |(x_1(T),\dots,x_p(T))|\ {\rm for}\  t\le T.$$
Choose $i\in \{ 1,\dots,p\}$  such that
 $x_i(T)=|(x_1(T),\dots,x_p(T))|$. By the definition of $T$, $x_i'(T)\ge 0$. 
 On the other hand, 
 \begin{equation}
 \begin{split}
 x_i'(T)&\le x_i(T)\bigg ( b_i-\mu_i x_i(T)+\sum_{j=1}^p |a_{ij}| \int_0^\infty K_{ij}(s) x_j(T-s)\, ds+|h_i(T)|\bigg)\\
& \le x_i(T)\bigg ( b_i-\mu_i x_i(T)+\sum_{j=1}^p |a_{ij}|x_i(T) +\de \bigg)\\
&=  x_i(T)\bigg[b_i-x_i(T)\Big(\mu_i -\sum_{j=1}^p |a_{ij}| \Big )+\de\bigg]<  x_i(T)\left [ b_i-(x_i(T)-1)\de\right]<0,
\end{split}
\end{equation}
which is not possible.
 This implies that  $(x_1(t),\dots,x_p(t))$ is bounded for $t\ge 0$, and from Lemma \ref{lem4.2} $X(t)$ is bounded on $[0,\infty)$ as well.
 \end{proof}

\begin{proof}[Proof of Theorem  \ref{thm4.1}] 
Since the $p\times p$ matrix $\hat{M}^{(p)}_0=[\de_{ij}\mu_i-|a_{ij}|]$ is a non-singular M-matrix, there are $(q_1,\dots ,q_p)>0, \eta=(\eta_1,\dots,\eta_p)>0$ such that \eqref{cond-eta-q} is satisfied with $n=p$ and $\tilde{a}_{ij}=a_{ij}$.
As already observed in Remark \ref{rmk3.1}, the scaling of the variables $x_i(t)\mapsto q_i^{-1}x_i(t), u_i(t)\mapsto q_i^{-1}u_i(t)$ allows us to consider \eqref{reduced-eq} with $\mu_i, a_{ij},c_i$ replaced by $\mu_iq_i,a_{ij}q_j, c_iq_i$, respectively, for $1\le i,j\le p$. In this way, we may take
  $q_1=\dots=q_p=1$ in  \eqref{cond-eta-q}, and  assume without loss of generality that 
 \eqref{eq-eta_tilde} holds with $n=p$ and $\tilde{a}_{ij}=a_{ij}$; i.e.,  there exists a positive vector $\eta=(\eta_1,\dots,\eta_p)$ such that
$\eta_i \mu_i>\sum_{j=1}^p \frac{1}{2}\Big (\eta_i |{a}_{ij}|+\eta_j |a_{ji}|\Big)$ for $ i=1,\ldots,p.$
 Hence, there is  $\vare_0>0$ sufficiently small such that for  $\vare\in (0,\vare_0)$ 
\begin{equation}\label{05}
\displaystyle{ \gamma_i= \gamma_i(\vare):=\eta_i \mu_i-\sum_{j=1}^p \frac{1}{2}\biggl(\eta_i |{a}_{ij}|+\eta_j |a_{ji}|\biggr)-
\frac{\vare\eta_i}{2}\bigg (c_i+\sum_{j=p+1}^n|a_{ij}|\bigg)>0, \q  i=1,\ldots,p}.
\end{equation}

Note that, from  hypothesis (ii) and Lemma \ref{lem4.1}(i) with $\al=2$,  
$$
\lim_{t \to \infty}\int_0^{\infty}K_{iq}(s)x_q^2(t-s)\, ds=0
$$
for any  $ i \in \{1,\ldots,n\}$ and $q\in\{ p+1,\dots,n\}$. Therefore, for any fixed $\vare>0$, it follows that
\begin{equation}\label{04}
h(t)=h_\vare(t):=\frac{1}{\vare}\sum_{i=1}^p\sum_{j=p+1}^n  \eta_i\frac{|{a}_{ij}|}{2}\int_0^{\infty}K_{ij}(s)x_j^2(t-s) ds\to 0\q {\rm as} \q 
t\to \infty.
\end{equation}

Take $E^*=E^{*,p}$ as in the statement of the theorem. 
%and suppose that $x_i^*\ge 0,1\le i\le p,$ 
%; otherwise, after reordering the variables, we take $E^*=E^{*,k}=(x_1^*,u_1^*,\dots, x_k^*,u_k^*,0,0\ldots,0,0)$ as in \eqref{Ep}, with $x_i^{*,k}>0$ for some $k<p$, and replace $p$ by $k$ in the computations below.  
Now, define (compare with \eqref{U1})
\begin{equation}\label{U1-red}
U_1(t)=\sum_{i=1}^p \eta_i \biggl(G(x_i(t),x_i^*)+\frac{c_i}{d_i}\frac{(u_i(t)-u_i^*)^2}{2}\biggr).
\end{equation}
 We now proceed as in the proof of Theorem \ref{thm3.1}, so some computations are omitted. Along positive solutions
$(x_1(t),u_1(t),x_2(t),u_2(t),\ldots,x_n(t),u_n(t))$ of system  \eqref{eq}, 
%we obtain
%\begin{equation*}
%%\hspace{-0.3cm}
%\begin{split}
%\dot U_1(t)&
%\leq \sum_{i=1}^p \eta_i \biggl\{-\mu_i (x_i(t)-x_i^*)^2-\frac{c_i e_i}{d_i}(u_i(t)-u_i^*)^2 \\
%&+\sum_{j=1}^p |{a}_{ij}| \int_0^{\infty}K_{ij}(s)|x_i(t)-x_i^*|\, |x_j(t-s)-x_j^*| ds \\
%&+\sum_{j=p+1}^n |{a}_{ij}| \int_0^{\infty}K_{ij}(s)|x_i(t)-x_i^*|\, x_j(t-s) ds\\
%&+c_i|x_i(t)-x_i^*|\int_0^{\infty}G_{i}(s)|u_i(t-s)-u_i(t)| ds\biggr\}. 
%\end{split}
%\end{equation*}
%Hence, 
for any fixed  $\vare\in (0,\vare_0)$ we get
\begin{equation}\label{U-1'}
%\hspace{-0.3cm}
\left.
\begin{array}{ll}\dot U_1(t)&
\displaystyle{\leq \sum_{i=1}^p \eta_i \biggl\{-\mu_i (x_i(t)-x_i^*)^2-\frac{c_i e_i}{d_i}(u_i(t)-u_i^*)^2} \\
&\displaystyle{+\sum_{j=1}^p \frac{|{a}_{ij}|}{2}\biggl( (x_i(t)-x_i^*)^2+\int_0^{\infty}K_{ij}(s)(x_j(t-s)-x_j^*)^2 ds\biggr)} \\
&\displaystyle{+\sum_{j=p+1}^n \frac{|{a}_{ij}|}{2}\biggl(\vare (x_i(t)-x_i^*)^2+\frac{1}{\vare}\int_0^{\infty}K_{ij}(s)x_j^2(t-s) ds\biggr)} \\
&\displaystyle{+\frac{c_i}{2} \biggl( \varepsilon (x_i(t)-x_i^*)^2+\frac{1}{\varepsilon} \int_0^{\infty}G_{i}(s)(u_i(t-s)-u_i(t))^2 ds\biggr) \biggr\}}.
\end{array}
\right.
\end{equation}
%Note that
%$$h(t):=\frac{1}{\vare}\sum_{i=1}^p\sum_{j=p+1}^n  \eta_i\frac{|{a}_{ij}|}{2}\int_0^{\infty}K_{ij}(s)x_j^2(t-s) ds\to 0\q {\rm as} \q 
%t\to \infty.$$
Now, set $U(t)=U_1(t)+U_2(t),t\ge 0,$ where
\begin{equation}
\left.
\begin{array}{ll}
\displaystyle{U_2(t)=\sum_{i=1}^p \eta_i \biggl(\sum_{j=1}^p \frac{|{a}_{ij}|}{2} \int_0^{\infty}K_{ij}(s) \int_{t-s}^t (x_j(u)-x_j^*)^2 du ds} \\
\ \ \ \ \ \ \ \ \ \ \ \displaystyle{+\frac{c_i}{2 \varepsilon} \int_0^{\infty}G_{i}(s) \int_{t-s}^t (u_i(u)-u_i(t))^2 du ds \biggr)}.
\end{array}
\right.
\end{equation}
Calculating the derivative of $U_2(t)$ along solutions as in \eqref{U-2} and adding up \eqref{U-1'}, from  \eqref{05} and \eqref{04} we obtain 
\begin{equation}\label{09}\dot{U}(t)
\leq -\sum_{i=1}^p \biggl(\gamma_i(x_i(t)-x_i^*)^2+\eta_i\frac{c_i e_i}{d_i}(u_i(t)-u_i^*)^2\biggr)+h(t)
\end{equation}
where $\gamma_i>0\, (1\le i\le p)$ and $h(t)\to 0$ as $t\to\infty$.
%$$\gamma_i:=\eta_i\mu_i-\frac{1}{2}\sum_{j=1}^p\biggl(\eta_i |{a}_{ij}|+\eta_j |a_{ji}|\biggr)-
%\frac{\vare\eta_i}{2}\bigg (c_i+\sum_{j=p+1}^n|a_{ij}|\bigg)>0,\q i=1,\dots ,p.$$

Next,  Lemma \ref{lem4.3}  allows us to conclude that $x_i(t),u_i(t)$  and  $x_i'(t),u_i'(t)$ are bounded  on $[0,\infty)$, therefore $x_i(t),u_i(t)$ are uniformly continuous on $[0,\infty)$. This and the assumptions imposed on the kernels $K_{ij},G_i$ imply that  $\dot U(t)$ and $h(t)$ are also uniformly continuous on $[0,\infty)$.
Write
$$f(t)=\sum_{i=1}^p \biggl(\gamma_i(x_i(t)-x_i^*)^2+\eta_i\frac{c_i e_i}{d_i}(u_i(t)-u_i^*)^2\biggr).$$We further define
$V(t)=U(t)-H(t),\ t\ge 0,$
where $H(t)=\int_0^t h(s)\, ds.$ From Lemma \ref{lem4.1}, $H(t)$ is bounded.
Clearly, $\dot V(t)=\dot U(t)-h(t)\le -f(t)\le 0$ on $[0,\infty)$, thus 
$V(t)\searrow c$ as $t\to\infty,$
for some $c\in\R$.
Since $\dot V(t)$ is uniformly continuous for $t\ge 0$, from Barbalat's lemma \cite[p.~5]{Gopal} , we conclude  that 
 $\lim_{t\to\infty} \dot V(t)=0,$
which implies that 
$\lim_{t\to\infty} f(t)=0.$
Thus, $x_i(t)\to x_i^*,\ u_i(t)\to u_i^*$ for $ i=1,\dots,p$, which ends the proof.
\end{proof}

\begin{rmk}\label{rmk4.1} {\rm In Theorem \ref{thm4.1}, it is possible that $x_i^*=0$ for some of the components $i\in\{1,\dots,p\}$ of the saturated equilibrium $E^*=(x_1^*,u_1^*,\dots, x_p^*,u_p^*,0,0\ldots,0,0)$. In this case,  after reordering the variables, $E^*$ takes the form $E^*=E^{*,k}=(x_1^*,u_1^*,\dots, x_k^*,u_k^*,0,0\ldots,0,0)$ as in \eqref{Ep}, with $x_i^{*,k}>0$ for some $k<p$, and, as in Theorem \ref{thm3.1},  in the definition of $\hat{M}^{(p)}_0$  one may replace $|a_{ij}|$ by $a_{ij}^-$ for $i,j=k+1,\dots ,p$.}
\end{rmk}
 
 Let $E^*=(x_1^*,u_1^*,\dots, x_p^*,u_p^*,0,0\ldots,0,0)$ (for some $p\in \{1,\dots,n-1\}$) be the saturated equilibrium of \eqref{eq}. In this situation, it is important to establish sufficient conditions for hypothesis (ii) in the statement of Theorem \ref{thm4.1} to be satisfied. 
 
With some additional conditions to \eqref{satu-cond-k} in Lemma \ref{lem2.3}, and 
based on the construction of a new Lyapunov functional (inspired however by the  approach in Hu {\it et al.} \cite{HTG}, Montes de Oca e P\'erez \cite{MOca:12} and Shi {\it et al.} \cite{Shi:2012}), we obtain a generalization of the result on partial extinction.

%%%%%%%%%%%%%%%%%%%%%%%%%%%%%%%% Theorem 3.1 %%%%%%%%%%%%%%%%%%%%%%%%%%%%%%%%%%%%
\begin{thm}\label{thm4.2}  Assume (H0) and that all solutions of  \eqref{eq} with initial conditions \eqref{initial} are bounded. For some $p\in\{1,2,\dots,n-1\}$ and some fixed $q \in\{p+1,\dots,n\}$, suppose that 
%$E^{*,p}=(x_1^*,u_1^*,x_2^*,u_2^*,\dots,x_p^*,u_p^*,0,0,\dots,0,0)$, 
 there exists a nonnegative vector $\al=(\al_1,\dots,\al_p)\ge 0$, such that
\begin{equation}\label{k-q-cond}
\left\{
\begin{split}
    &\sum_{i=1}^p\al_ib_i-b_q>0\\
    &\sum_{i=1}^p\al_i (\de_{ij}\la_i+a_{ij})-a_{qj}\le 0,\q j=1,2,\dots,p\\
     &\sum_{i=1}^p\al_i a_{ij}-(\de_{qj}\la_q+a_{qj})\le 0,\q j=p+1,\dots,n
    \end{split}
    \right .
 \end{equation}
 If $\al=0$, in addition suppose that $\mu_q-a_{qq}^->0$.
Then, any positive solution $(x_1(t),u_1(t),\dots ,x_n(t),u_n(t))$ of \eqref{eq} satisfies
$x_q=O(e^{-\eta t})$ as $t\to \infty$ and some $\eta >0$; in particular,  $x_q\in L^2[0,\infty)$
and
$\displaystyle\lim_{t \to \infty}x_q(t)=0$.
\end{thm}

\begin{proof}
Fix $q >p$, and assume that for some non-negative constants   $\al_1,\dots,\al_p$ conditions  \eqref{k-q-cond} are satisfied. If $\al_1=\dots=\al_p=0$ and $\mu_q-a_{qq}^->0$, from \eqref{k-q-cond}  we have $b_q<0$ and  all the entries  $\de_{qj}\la_q+a_{qj}$ of the $q$-line of $M$ are nonnegative, therefore 
$$x_q'(t)\le x_q(t)\Big (b_q-\mu_q x_q(t)+a_{qq}^- \int_0^{\infty}K_{qq}(s)x_q(t-s)\, ds\Big ),$$ and the result follows by Remark \ref{rmk3.2} applied with $n=1$.

With at least some  $\al_i>0$,  consider the following Lyapunov functional:
\begin{equation}\label{V}
\left.
\begin{array}{ll}
V_{p,q}(t)&\displaystyle{=x_1^{-{\alpha}_1}(t) x_2^{-{\alpha}_2}(t)\ldots x_p^{-{\alpha}_p}(t) x_q(t)
\times \exp \Biggl\{\sum_{i=1}^{p}{\alpha}_i \biggl(\frac{c_i}{e_i} u_i(t)} \\
& \displaystyle{+\sum_{j=1}^n a_{ij} \int_0^{\infty} K_{ij}(s)\int_{t-s}^t x_j(\theta)\, d \theta ds+c_i \int_0^{\infty} G_i(s) \int_{t-s}^t u_i(\theta) \, d \theta ds\biggr)} \\
&  \displaystyle{-\biggl(\frac{c_q}{e_q} u_q(t)+\sum_{j=1}^{n} a_{qj} \int_0^{\infty} K_{qj}(s) \int_{t-s}^t x_j(\theta) \, d \theta ds} \\
&  \displaystyle{+c_q \int_0^{\infty} G_{q}(s) \int_{t-s}^t u_q(\theta) \, d \theta ds\biggr)\Biggr\} }. 
\end{array}
\right.
\end{equation}
For any positive solution $X(t)=(x_1(t),u_1(t),\dots, x_n(t),u_n(t))$ of \eqref{eq},
calculating the derivative of $V_{p,q}(t)$ with respect to $t>0$ along $X(t)$, we have 
\begin{equation*}
\left.
\begin{array}{ll}
\dot{V}_{p,q}(t)&= \displaystyle{V_{p,q}(t)\Biggl\{-\sum_{i=1}^{p}\al_i\biggl(b_i-\mu_i x_i(t)} \\
&\displaystyle{-\sum_{j=1}^n a_{ij}\int_0^{\infty} K_{ij}(s) x_j(t-s) d s-c_i \int_0^{\infty} G_i(s) u_i(t-s) ds \biggr)} \\
& \displaystyle{+\biggl(b_q-\mu_q x_q(t)-\sum_{j=1}^n a_{qj} \int_0^{\infty} K_{qj}(s) x_j(t-s) ds-c_q \int_0^{\infty} G_{q}(s) u_{q}(t-s) ds \biggr)} \\
& \displaystyle{
+\sum_{i=1}^p {\alpha}_i\Biggl [\frac{c_i}{e_i}(-e_i u_i(t)+d_i x_i(t))+\sum_{j=1}^n a_{ij} \int_0^{\infty} K_{ij}(s)[x_j(t)-x_j(t-s)]\, ds} \\
& \displaystyle{+c_i \int_0^{\infty} G_i(s)[u_i(t)-u_i(t-s)] ds\Biggr]} \\
& \displaystyle{-\biggl(\frac{c_q}{e_q}(-e_q u_q(t)+d_q x_q(t))+\sum_{j=1}^n a_{qj} \int_0^{\infty} K_{qj}(s)[x_j(t)-x_j(t-s)]\, ds} \\
& \displaystyle{+c_{q} \int_0^{\infty} G_{q}(s) \{u_{q}(t)-u_q(t-s)\} ds\biggr)\Biggr\}} ,
\end{array}
\right.
\end{equation*}
thus
\begin{equation}\label{DV}
\left.
\begin{array}{ll}
\dot{V}_{p,q}(t)&\displaystyle{=V_{p,q}(t)\biggl\{-\sum_{i=1}^{p}\al_i\biggl(b_i-\mu_i x_i(t)-\sum_{j=1}^n a_{ij} x_j(t)-\frac{c_i d_i}{e_i} x_i(t)\biggr)} \\
&\displaystyle{+\biggl(b_q-\mu_q x_q(t)-\sum_{j=1}^n a_{qj}  x_j(t)-\frac{c_q d_q}{e_q} x_q(t)\biggr)\biggl\}} \\
&\displaystyle{=V_{p,q}(t)\biggl\{-\sum_{i=1}^{p}\al_i\biggl(b_i-\sum_{j=1}^n (\delta_{ij} \lambda_i+a_{ij}) x_j(t)\biggr)+\biggl(b_q-\sum_{j=1}^n (\delta_{qj} \lambda_q+a_{qj}) x_j(t) \biggr)\biggr\}} \\
&\displaystyle{=V_{p,q}(t)\biggl\{-\biggl(\sum_{i=1}^p\al_i b_i-b_q\biggr)+\sum_{j=1}^p \biggl( \sum_{i=1}^p\al_i(\delta_{ij} \lambda_i+a_{ij})
-a_{qj}\biggr)x_j(t)} \\
& \displaystyle{+\sum_{j=p+1}^n \biggl( \sum_{i=1}^p\al_i a_{ij}-(\delta_{qj} \lambda_q+a_{qj})\biggr)x_j(t)\biggr\}}. 
\end{array}
\right.
\end{equation}
Define  $$\eta:=\sum_{i=1}^{p}\al_i b_i-b_q >0.$$
By \eqref{k-q-cond} and \eqref{DV}, we obtain 
${\dot{V}_{p,q}(t)} \leq -\eta V_{p,q}(t)$,
thus 
\begin{equation}\label{upperbddV}
V_{p,q}(t) \leq V_{p,q}(0)e^{-\eta t}.
\end{equation}
On the other hand, since the positive kernels $K_{ij}(t), G_j(t)$ satisfy  \eqref{1.3} and all coordinates $x_j(t),u_j(t), 1\le j\le n$, of the solution $X(t)$ are bounded from above by some positive constant $C$, we get
\begin{equation}\label{lowerbddV}
\begin{split}
V_{p,q}(t)&\ge C^{-(\al_1+\cdots +\al_p)} x_q(t) \\
&\times \exp \biggl\{-C\biggl[\sum_{j=1}^n\biggl(\sum_{i=1}^p  \al_i|a_{ij}|\int_0^{\infty} s K_{ij}(s) \, ds+|a_{qj}| \int_0^{\infty} s K_{qj}(s) \, ds\biggr)\\
&+\frac{c_q}{e_q} +c_q \int_0^{\infty} s G_{q}(s) \, ds\biggr]\biggr\}\ge  C_1 x_q(t),
\end{split}
\end{equation}
for some constant $C_1>0$. From the upper and lower estimates \eqref{upperbddV},  \eqref{lowerbddV}, we obtain $x_q(t)\le C_1^{-1} V_{p,q}(0)e^{-\eta t}$.
\end{proof} 

%\begin{rmk}\label{rmk4.1} In Theorem \ref{thm4.2}, the components of $x_i^*$ need not be positive, nor the constants $\al_i\, (1\le i\le p)$. 
%Note also that, if $b_q<0$ and all the entries of the $q$-line of $M$ are nonnegative, conditions \eqref{k-q-cond} are satisfied with $\al_1=\dots =\al_p=0$. 
%\end{rmk}
%
\begin{rmk}\label{rmk4.2} {\rm Suppose that $M$ is a P-matrix and that there exists $p\in \{1,\dots,n-1\}$ such that
the unique saturated equilibrium has the form 
$E^{*,p}=(x_1^*,u_1^*,x_2^*,u_2^*,\dots,x_p^*,u_p^*,0,0,\dots,0,0),$
with  $x_i^*\ge 0$ and $b_i=\sum_{j=1}^p (\la_i \de_{ij}+a_{ij})x_j^*$ for $ i=1,\dots,p,$ and $b_q<\sum_{j=1}^p a_{pj}x_j^*$ for $q=p+1,\dots,n$. Define $x^*=(x_1^*,\dots,x_p^*,0,\dots,0)$. In this writing,  for each $i=1,\dots,p$, we now  let   $x_i^*= 0$ if $(Mx^*)_i=b_i$, but  in \eqref{2.1} we demand $(Mx^*)_q>b_q$ for $q=p+1,\dots,n$. Proceeding as in Section 2,  instead of \eqref{satu-cond-k},  in this situation we have 
\begin{equation*}
\left\{
\begin{array}{ll}
R_i^p\ge 0, \ & \mbox{for any} \ i=1,\ldots,p \\
R_q^{p+1,q} < 0, &\mbox{for any} \ q=p+1,\dots,n,
\end{array}
\right.
\end{equation*}
and the strict inequality in \eqref{b-q} implies now that for each $q=p+1,\dots,n$ the first conditions in \eqref{k-q-cond} are satisfied with $\al_i^q=|a_{qi}|x_i^*,\, i=1,\dots,p$.}
\end{rmk}

From Theorems \ref{thm4.1} and \ref{thm4.2},  rather than Theorem \ref{thm3.1}, an alternative outcome  for global attractivity with partial extinction is as follows:

% it follows that if a saturated equilibrium on the boundary of $R^{2n}_+$ is a global attractor for the associated reduced system 
\begin{thm}\label{thm4.3}  Assume (H0) and  that $M$ is a P-matrix. Suppose that  the saturated equilibrium is neither positive nor trivial,  so it has the form 
$E^{*,p}=(x_1^*,u_1^*,x_2^*,u_2^*,\dots,x_p^*,u_p^*,0,0,\dots,0,0)\ne 0$  for some $p\in\{1,2,\dots,n-1\}$, with   $x_i^*\ge 0$ and $b_i=\sum_{j=1}^p a_{ij}x_j^*,\, i=1,\dots,p$.
In addition, assume that:\vskip 0cm
(i) the $n\times n$ matrix $M_0^-$ and the $p\times p$ matrix $\hat M_0^{(p)}=[\de_{ij}\mu_i-|a_{ij}|]$ are non-singular M-matrices;\vskip 0cm
(ii) there exist nonzero vectors $\al^{(q)}=(\al_1^{(q)},\dots,\al_p^{(q)})\ge 0$ such that conditions \eqref{k-q-cond} are satisfied for $q=p+1,\dots,n$.\vskip 0cm
Then $E^{*,p}$ is a global attractor of all positive solutions of \eqref{eq}.
%where  $x_i^*=R_i^p/R_0^p, u_i^*=\frac{d_i}{e_1}x_i^*,\, i=1,\dots,p,$
In other words, with $x_i^*=R_i^p/R_0^p\, (1\le i\le p)$, any positive solution $(x_1(t),u_1(t),\dots ,x_n(t),u_n(t))$  of \eqref{eq} satisfies
\begin{equation*}
\left\{
\begin{array}{ll}
\lim_{t \to \infty}x_i(t)=x_i^*&\ {\rm for}\ j=1,\dots,p,\\ \lim_{t \to \infty}x_q(t)=0&\ {\rm for}\ q=p+1,\dots,n.\end{array}
\right.
\end{equation*}
\end{thm}

 Another consequence of Theorems \ref{thm4.1} and \ref{thm4.2} is given below.

%\begin{cor}\label{cor4.2}  Assume (H0), that $M$ is a P-matrix and  $M_0^-$  a non-singular M-matrix.  In addition, suppose that:\vskip 0cm
%(i) $b_1>0$ and  $\mu_1>|a_{11}|$;\vskip 0cm
%(ii)  $R_q^{2,q}=(\la_1+a_{11})b_q-a_{q1}b_1<0$ for $q=2,\dots,n$;\vskip 0cm
%(iii) $a_{q1}\ge 0$ for $q=2,\dots,n$;\vskip 0cm
%(iv) $a_{q1}a_{1j}\le a_{qj}(\la_1+a_{11})$ for $q,j=2,\dots,n, q\ne j$.\vskip 0cm
%Then the saturated equilibrium is
%$E^{*,1}=(x_1^*,u_1^*,0,0,\dots,0,0),$
%with  $x_1^*=b_1/(\la_1+a_{11}), u_1^*=\frac{d_1}{e_1}x_1^*$, and $E^{*,1}$
% is a global attractor of all positive solutions of \eqref{eq}.
%\end{cor}

\begin{thm}\label{thm4.4}  Assume (H0), that $M$ is a P-matrix and  $M_0^-$  a non-singular M-matrix.  Suppose also that
\begin{equation}\label{n,p=1}
a_{q1}\ge 0\q {\rm and}\q R_q^{2,q}:=(\la_1+a_{11})b_q-a_{q1}b_1<0\q {\rm for}\q q=2,\dots,n.
\end{equation}
Then $E^{*,1}=(x_1^*,u_1^*,0,0,\dots,0,0),$
with  $x_1^*=b_1/(\la_1+a_{11}), u_1^*=\frac{d_1}{e_1}x_1^*$, is  the saturated equilibrium of \eqref{eq}.
 Moreover, $x_q\in L^2[0,\infty)$ and $x_q(t)\to 0$ as $n\to\infty$, for all $q=2,\dots,n$ and all positive solutions $X(t)=(x_1(t),u_1(t),\dots ,x_n(t),u_n(t))$ of \eqref{eq}.\vskip 0cm
If in addition $b_1>0$ and  $\mu_1>|a_{11}|$,
  then $E^{*,1}$
 is a global attractor of all positive solutions.
\end{thm}

\begin{proof} Since $M_0^-$ is a non-singular M-matrix, all solutions of the initial value problems \eqref{eq}-\eqref{initial} are bounded.  From Lemma \ref{lem2.3}, conditions \eqref{n,p=1} imply  that $E^{*,1}$ as above is the saturated equilibrium. 

In a first step, take $p=n-1$ and $q=n$  in Theorem \ref{thm4.2}. With $\al^{(n)}=(\al_{n1},0,\dots,0)\in\R^{n-1}$ and $\al_{n1}\ge 0$, conditions \eqref{k-q-cond} are equivalent to
\begin{equation*}
\left\{
\begin{split}
    &\al_{n1} b_1>b_n\\
    &\al_{n1} (\la_1+a_{11})\le a_{n1}\\
     &\al_{n1} a_{1n}\le \la_n+a_{nn}.
    \end{split}
    \right .
 \end{equation*}
Choose $\al_{n1}=\frac{a_{n1}}{\la_1+a_{11}}\ge 0$. From \eqref{n,p=1}, the first two conditions  are satisfied. On the other hand, 
$$\al_{n1} a_{1n}-(\la_n+a_{nn})=\frac{a_{n1}a_{1n}-(\la_1+a_{11})(\la_n+a_{nn})}{\la_1+a_{11}}=-
\frac{1}{\la_1+a_{11}}
\left |
\begin{array}{cc}
\la_1+a_{11}& a_{1n} \\
a_{n1}& \la_q+a_{nn}
\end{array}\right |<0
$$
because $M$ is a P-matrix. We now use a Lyapunov functional $V_{n-1,n}(t)$ as in  \eqref{V}, see the proof of Theorem \ref{thm4.2}. For solutions $X(t)=(x_1(t),u_1(t),\dots ,x_n(t),u_n(t))$ of the IVPs \eqref{eq}-\eqref{initial}, we deduce that  ${\dot{V}_{n-1,n}(t)} \leq -\eta_n V_{n-1,n}(t)$,  for some $\eta_n>0$, from which it follows that $x_n(t)=O(e^{-\eta_n t})$  as $t\to\infty$.
 
In a next step, we show that
 system \eqref{eq} can be reduced to  \eqref{reduced-eq} with $p=n-1$. 
 Take  $p=n-2$ and $q=n-1$  in Theorem \ref{thm4.2}, and choose $\al^{(n-1)}=(\al_{n-1,1},0,\dots,0)\in\R^{n-2}$
 with $\al_{n-1,1}=\frac{a_{n-1,1}}{\la_1+a_{11}}\ge 0$. As before, we obtain
 \begin{equation*}
\left\{
\begin{split}
    &\al_{n-1,1} b_1>b_{n-1}\\
    &\al_{n-1,1} (\la_1+a_{11})= a_{n-1,1}\\
     &\al_{n-1,1} a_{1,n-1}< \la_{n-1}+a_{n-1,n-1}.
    \end{split}
    \right .
 \end{equation*}
 Proceeding as  in the the proof of Theorem \ref{thm4.2}, for $V_{n-1,n-2}(t)$ defined by \eqref{V} and calculating the derivative along solutions $X(t)$ of  \eqref{eq}-\eqref{initial},
 from \eqref{DV} we now obtain an estimate of the form
  $${\dot{V}_{n-1,n-2}(t)} \leq \big (-\eta+h_n(t)\big ) V_{n-1,n-2}(t),$$
 where $\eta=\al_{n-1,1} b_1-b_{n-1}>0$ and $h_n(t)=(\al_{n-1,1}-a_{n-1,n})x_n(t)=O(e^{-\eta_n t})$  as $t\to\infty$. Arguing as in \eqref{upperbddV}, \eqref{lowerbddV}, we therefore conclude that $x_{n-1}(t)=O(e^{-\eta_{n-1} t})$ as $t\to\infty$, for some $\eta_{n-1}>0$. 
 
 Recursively,  in this way system \eqref{eq} is reduced to 
\begin{equation}\label{reduced,n=1}
\left\{
\begin{array}{ll}
\displaystyle{x_1'(t)=x_1(t)\biggl(b_1-\mu_1 x_1(t)-a_{11}\int_0^{\infty}K_{11}(s)x_1(t-s)\, ds-c_1 \int_0^\infty G_1(s)u_1(t-s)\, ds\biggr)} \\
\displaystyle{u_1'(t)=-e_l u_1(t)+d_1 x_1(t)}
\end{array}
\right. 
\end{equation}
By virtue of Theorem \ref{thm4.1}, the equilibrium $E^{*,1}$  is a global attractor for \eqref{eq} if $\mu_1>|a_{11}|$.
\end{proof}

For $n=2$, we obtain the following corollary.

%%%%%%%%%%%%%%%%%%%%%%%%%%%%%%%% Corollary 3.1 %%%%%%%%%%%%%%%%%%%%%%%%%%%%%%%%%%%%
\begin{cor}\label{cor4.1} 
Consider \eqref{eq} with $n=2$,
and suppose that $M$ is a P-matrix, that is, conditions \eqref{M0-matrix-2} are satisfied. If 
\begin{equation}\label{hatM-matrix-2}
\mu_i-a_{ii}^->0\ (i=1,2), \q a_{21}\ge 0\quad \mbox{and} \quad a_{21} b_1>(\la_1+a_{11})b_2, 
\end{equation}
then any solution of \eqref{eq} with initial conditions \eqref{initial} satisfies $ \lim_{t \to \infty}x_2(t)=0.$
%\begin{equation}\label{x-2-0}
%\displaystyle\lim_{t \to \infty}x_2(t)=0. \end{equation} 
Then, \eqref{eq} is reduced to \eqref{reduced,n=1}, 
whose equilibrium $(x_1^{*,1},u_1^{*,1})$ given by \eqref{E1,n=2} is positive and globally  stable if 
\begin{equation}\label{x-1}
b_1>0 \quad \mbox{and} \quad \mu_1>|a_{11}|. 
\end{equation}
In this case, $(x_1^{*,1},u_1^{*,1},0,0)$ is a global attractor for \eqref{eq}.
\end{cor}

\begin{proof} Assume  \eqref{hatM-matrix-2}. With the notation in Example \ref{exmp2.1}, $E^{*,1}:=(x_1^{*,1},u_1^{*,1},0,0)$ is the saturated equilibrium. With  \eqref{hatM-matrix-2},
 $\det M_0^-=(\mu_i-a_{11}^-)(\mu_2-a_{22}^-)>0$, thus in particular $M_0^-$ is a non-singular M-matrix, hence all positive solutions of \eqref{eq} are bounded.    With the additional hypotheses  \eqref{x-1}, Theorem \ref{thm4.4} yields that  $E^{*,1}$ attracts all  positive solutions of \eqref{eq}.
\end{proof} 

\begin{rmk}\label{rmk4.3} {\rm Let $n=2$, and $M$ be a P-matrix.  If $a_{21}\ge 0$ and $(Mx^*)_2>b_2$, i.e., with the strict inequality $a_{21} b_1>(\la_1+a_{11})b_2$, Corollary \ref{cor4.1} 
%hows that $\mu_i-a_{ii}^->0\ (i=1,2)$ are sufficient conditions for the extinction of the second population $x_2(t)$, whereas $ \mu_1-|a_{11}|>0, \mu_2-a_{22}^->0$ imply the global attractivity of $(x_1^{*,1},u_1^{*,1},0,0)$. Therefore, in this situation Corollary \ref{cor4.1} 
provides a better result than Corollary \ref{cor3.1}. It also improves the result of Li {\it et al.} \cite{Li:2013}, where the sufficient conditions for attractivity depend on the controls.}
\end{rmk}

\begin{rmk}\label{rmk4.4} {\rm Similarly to what was done in Theorem \ref{thm4.4} for $p=1$, following a recursive scheme, one could provide sufficient conditions to have all the last $2(n-p)$ components $x_q(t),u_q(t)$ of solutions, with $p= 2,3,\dots,n-1$, satisfying $x_q(t)\to 0$ as $t\to\infty$.}
\end{rmk}

\section{Examples}
\setcounter{equation}{0}

For simplicity, we present some examples with $n=2$. Simpler versions of Examples \ref{exmp5.1} and \ref{exmp5.3} were given in \cite{Faria:2015}, but here they are  revisited in light of the better criteria in this paper.

\begin{exmp}\label{exmp5.1} {\rm Consider a planar version  of \eqref{eqUn} with e.g. discrete delays $\tau_{ij}\ge 0$:
\begin{equation}\label{5.2}
\left\{
\begin{split}
\displaystyle x_1^{\prime}(t)&=x_1(t)\Bigl(b_1-\mu_1x_1(t)-a_{11} x_1(t-\tau_{11})-  a_{12} x_2(t-\tau_{12})\Bigr)\\
%\displaystyle u_1^{\prime}(t)&=-e_1 u_1(t)+d_1x_1(t)\cr
\displaystyle x_2^{\prime}(t)&=x_2(t)\Bigl(b_2-\mu_2x_2(t)-a_{21} x_1(t-\tau_{21})-a_{22} x_2(t-\tau_{22})
\Bigr) \\
%\displaystyle u_2^{\prime}(t)&=-e_2 u_2(t)+d_2 x_2(t),\cr
\end{split}
\right.
\end{equation}
Choosing $b_1=1,b_2=-{5\over 4}, \mu_1=\mu_2=1, a_{11}=a_{22}={1\over 2},a_{12}= {1\over 8}, a_{21}=-2$, 
we obtain a predator-prey system,
with community matrix $M_0= \left [\begin{array}{cc}3/2& 1/8\\ -2&3/2\end{array}\right ]$.
%$$M_0=\left [\begin{array}{cc} \mu_1+a_{11}&a_{12}\\ a_{21}&\mu_2+a_{22}\end{array}\right ]
%= \left [\begin{array}{cc} 3/2& 1/8\\ -2&3/2\end{array}\right ]$$
For this system,  $(x_1^*,x_2^*)=({{53}\over {80}}, {1\over {20}})$ is the positive equilibrium.
 Moreover, since $\det M_0>0$ and 
$\hat M_0= \left [\begin{array}{cc} 1/2& -1/8\\ -2&1/2\end{array}\right ]$
is an { M-matrix}, from \cite{Faria:2010} it follows that  $(x_1^*,x_2^*)$ is globally atractive.
We now introduce  controls, with the purpose of driving the predators to extinction.

For nonnegative coefficients $c_i$ and positive $d_i,e_i$, consider the corresponding system obtained  by adding delayed    terms with controls $-c_ix_i(t)u_i(t-\sigma_i)\, (\sigma_i\ge 0)$ to each equation $i$ in \eqref{5.2}, as well as the equations for the control variables $u_i'(t)=-e_iu_i(t)+d_ix_i(t),\, i=1,2$, as in \eqref{eq}. Note that
  $ b_2(\mu_1+a_{11}+c_1{{d_1}\over {e_1}})\le a_{21}b_1$ if and only if $c_1{{d_1}\over {e_1}}\ge {1\over {10}}$ (and any $c_2\ge 0$), in which case
 $ E^{*,1}:= (x_1^*,u_1^*,x_2^*,u_2^*)=\Big({1\over {{3\over 2}+c_1{{d_1}\over {e_1}}}},  {{d_1e_1}\over {{3\over 2}e_1+c_1d_1}},0,0\Big )$ is the saturated equilibrium; moreover, from Corollary \ref{cor3.1}, $E^{*,1}$  is a global attractor of all positive solutions without any further restrictions. For this particular situation, the global attractivity of  $ E^{*,1}$ was derived in \cite{Faria:2015} under the more restrictive  condition
  ${1\over {10}}\le c_1{{d_1}\over {e_1}}\le{1\over 2}.$
  Of course, similarly to the above situation,  finite distributed or infinite delays may have been considered.}
  \end{exmp}

  \begin{exmp}\label{exmp5.2} {\rm Consider a planar competitive system without controls given by \eqref{5.2}, with $b_1=2, \mu_1=3-a, b_2=\mu_2=1, a_{11}=a, a_{12}=4, a_{21}=a_{22}=2\ (a<3)$. Hence,  $M_0= \left [\begin{array}{cc} 3 & 4\\ 2&3\end{array}\right ]$ is the community matrix and $x^*=(2/3,0)$ is the saturated equilibrium. Since $a_{21}b_1>(\mu_1+a_{11})b_2$,   Corollary \ref{cor4.1} implies that all positive solutions $x(t)=(x_1(t),x_2(t))$ satisfy $x_2(t)\to 0$ as $t\to \infty$. One easily verifies that the characteristic equation for the linearized equation about $x^*$ is given by
  $$h(\la)=0,\q {\rm for}\q h(\la)=(\la+\frac{1}{3})\Big(\la +\frac{2}{3}(3-a)+\frac{2}{3}ae^{-\la \tau_{11}}\Big).$$
  Therefore, if $a\le 3/2$,  $x^*$ is locally asymptotically stable (see e.g.~\cite{HaleLunel}); moreover, if $0<a<3/2$, we have $\mu_1>|a_{11}|$ and  Corollary \ref{cor4.1}  yields now that $x^*$ is globally attractive.
  
  We now introduce a control variable $u_1(t)$, so that a system of the form \eqref{eq} with $n=2$ and  $ c_1,d_1,e_1>0, c_2=0$ is obtained, with the purpose of keeping the $x_2(t)$ population extinct with time, but trying to stabilize the $x_1(t)$ population at a level lower than $x_1^*=2/3$.
  If $c_1d_1<e_1$, then conditions \eqref{hatM-matrix-2} still hold, thus the  saturated equilibrium is $ E^{*,1}:=(x_1^*,u_1^*,0,0)$ with $x_1^*=2/(3+c_1\frac{d_1}{e_1})$ and $x_2(t)\to 0$ as $t\to\infty$ for any positive solution $(x_1(t),u_1(t),x_2(t),u_2(t))$ of the controlled system. If we still suppose that  $0<a<3/2$, 
  then $ E^{*,1}$ is a global attractor.}
    \end{exmp}
    
    \begin{exmp}\label{exmp5.3} {\rm Consider the uncontrolled system \eqref{eqUn}  with $n=2$ and take e.g.~ $b_1=1,b_2={1\over 3}, \mu_1=\mu_2=1, a_{11}= a_{21}={1\over 2}$, arbitrary  coefficients $a_{12}, a_{22}\in \R$ and $K_{ij}(s)=\ga e^{-\ga s}$ for some $\ga >0$.
%\begin{equation}\label {ex5.3}
%\left\{
%\begin{split}
%\displaystyle x_1^{\prime}(t)&=x_1(t)\Bigl(1-x_1(t)- {1\over 2} \int_0^{\infty} \ga e^{-\ga s}x_1(t-s)\, ds-  a_{12}  \int_0^{\infty} \ga e^{-\ga s}x_2(t-s)\, ds\Bigr) \\
%%\displaystyle u_1^{\prime}(t)&=-e_1 u_1(t)+d_1x_1(t)\cr
%\displaystyle x_2^{\prime}(t)&=x_2(t)\Bigl({1\over 3}-x_2(t)-{1\over 2} \int_0^{\infty} \ga e^{-\ga s}x_1(t-s)\, ds- a_{22} \int_0^{\infty} \ga e^{-\ga s}x_2(t-s)\, ds
%\Bigr). \\
%\end{split}
%\right.
%\end{equation}
 With the above notations, we have
$$M_0= \left [\begin{array}{cc}3/2& a_{12}\\ 1/2&1+a_{22}\end{array}\right ],
\ \hat M_0= \left [\begin{array}{cc} 1/2& -|a_{12}|\\ -1/2&1-|a_{22}|\end{array}\right ],
%\ {\cal M}_0^{(1)}= \left [\begin{array}{cc} 1/2& -|a_{12}|\\ -1/2&1-a_{22}^-\end{array}\right ],
\  M_0^-= \left [\begin{array}{cc} 1& -a_{12}^-\\ 0&1-a_{22}^-\end{array}\right ].$$ 
One easily sees that $(X_1,0)=({2\over 3},0)$ is a  saturated equilibrium.
If $a_{22}>-1$, conditions \eqref{hatM-matrix-2} and \eqref{x-1} are satisfied with $\la_1=0$. Applying Corollary \ref{cor4.1} to this system without controls, we deduce that $({2\over 3},0)$ is a global attractor of all its positive solutions.

We now introduce the controls, in order to recover the $x_2(t)$ population, which otherwise would be lead to extinction. For positive coefficients $c_i,d_i,e_i$, denote $\al_i:=c_i{{d_i}\over {e_i}},\, i=1,2,$ consider the corresponding system with controls as in \eqref{eq}, 
%\begin{equation}\label{5.1'-control}
%\left\{
%\begin{split}
%\displaystyle x_1^{\prime}(t)&=x_1(t)\Bigl(1-x_1(t)- {1\over 2} \int_0^{\infty} \ga e^{-\ga s}x_1(t-s)\, ds-  a_{12}  \int_0^{\infty} \ga e^{-\ga s}x_2(t-s)\, ds\\
%&-c_1 \int_0^{\infty} \ga e^{-\ga s}u_1(t-s)\, ds\Bigr) \\
%u_1'(t)&=-e_1u_1(t)+d_1x_1(t)\\
%%\displaystyle u_1^{\prime}(t)&=-e_1 u_1(t)+d_1x_1(t)\cr
%\displaystyle x_2^{\prime}(t)&=x_2(t)\Bigl({1\over 3}-x_2(t)-{1\over 2} \int_0^{\infty} \ga e^{-\ga s}x_1(t-s)\, ds- a_{22} \int_0^{\infty} \ga e^{-\ga s}x_2(t-s)\, ds\\
%&-c_2 \int_0^{\infty} \ga e^{-\ga s}u_2(t-s)\, ds
%\Bigr), \\
%u_2'(t)&=-e_2u_2(t)+d_2x_2(t),\\
%\end{split}
%\right.
%\end{equation}
and the controlled matrix $M$  given by
$M= \left [\begin{array}{cc} 3/2+\al_1\ & a_{12}\\ 1/2&1+\al_2+a_{22}\end{array}\right ].$
For any values of $\al_1,\al_2$ such that $\det M>0$,  the controlled system now has a positive equilibrium $E^*=(x_1^*,u_1^*,x_2^*,u_2^*)$, with 
$$x_1^*=(\det M)^{-1}[1+\al_1+a_{22}-\frac{1}{3}a_{12}]\q {\rm and}\q x_2^*=(\det M)^{-1}\frac{\al_1}{3}.$$
If %we replace conditions \eqref{X1} by the stronger requirement 
$1-|a_{22}|>|a_{12}|,$
then $\hat M_0$ is a non-singular M-matrix, and Theorem \ref{thm3.2} yields that $E^*$ is a global attractor of all positive solutions. For instance, for the case of $a_{22}={1\over 2}$ and  $|a_{12}|<{1\over 2}$,  $E^*$ is always globally attractive for the system with controls. This example generalizes and improves the situation considered in \cite[Example 5.1]{Faria:2015}, where the coefficients  $a_{12},a_{22}$ were chosen to be $a_{22}={1\over 2}$ and  $a_{12}={1\over 8}$, and the global attractivity of $E^*$ was derived only if $\al_i\le 1/4, i=1,2$.}
\end{exmp}

\section*{Acknowledgements}   Research  supported by Scientific Research (c), No. 15K05010 of Japan Society for the Promotion of Science (Yoshiaki Muroya) and by
 Funda\c c\~ao para a Ci\^encia e a Tecnologia (Portugal) under project UID/MAT/\-04561/2013 (Teresa Faria).
 We thank the referees for their careful reading and valuable comments, which have led  to  improvements in the writing of some proofs.
 %to  improvements in the exposition of some arguments.

%%%%%%%

%%%%%%%%%%%%%%%%%%%%%%%%%%%%%% 

\end{document}